\documentclass[a4paper]{amsart}
\usepackage{graphicx}
\usepackage[bookmarksnumbered,colorlinks,plainpages]{hyperref}
\usepackage{extarrows}
\usepackage{float}
\usepackage{amsthm}
\usepackage[all,2cell]{xy}
\usepackage{amsfonts}
\usepackage{amsmath}
\usepackage{amssymb}
\usepackage{xypic,leqno,amscd,amssymb,pstricks,latexsym,amsbsy,xypic,mathrsfs,verbatim}
\usepackage{multirow}
\usepackage{booktabs}
\usepackage{makecell}
\usepackage{cases}
\usepackage{appendix}
\usepackage{graphicx,colortbl}
\usepackage{tikz-cd}
\usepackage{etex}
\usepackage{mathtools}
\usepackage[bookmarksnumbered,colorlinks,plainpages]{hyperref}
\hypersetup{
	colorlinks=true,
	citecolor=blue
}
\newtheorem{theorem}{Theorem}[section]

\newtheorem{corollary}[theorem]{Corollary}
\newtheorem{lemma}[theorem]{Lemma}
\newtheorem{proposition}[theorem]{Proposition}
\newtheorem{definition-proposition}[theorem]{Definition-Proposition}

\newtheorem{question}[theorem]{Question}

\theoremstyle{definition}
\newtheorem{definition}[theorem]{Definition}

\newtheorem{remark}[theorem]{Remark}
\newtheorem{example}[theorem]{Example}

\numberwithin{equation}{section}

\renewcommand{\AA}{\mathbb{A}}
\newcommand{\DD}{\mathbb{D}}
\newcommand{\EE}{\mathbb{E}}
\renewcommand{\L}{\mathbb{L}}
\newcommand{\X}{\mathbb{X}}

\newcommand{\Z}{\mathbb{Z}}

\renewcommand{\c}{\vec c}
\newcommand{\vdelta}{\vec{\delta}}
\newcommand{\y}{\vec y}
\newcommand{\x}{\vec x}

\newcommand{\vell}{\vec {\ell}}

\newcommand{\z}{\vec z}

\newcommand{\s}{\vec s}

\newcommand{\w}{\vec{\omega}}

\newcommand{\cut}{\ar@{-}@[|(5)]}

\newcommand{\Hom}{\operatorname{Hom}\nolimits}

\newcommand{\End}{\operatorname{End}\nolimits}
\newcommand{\Ext}{\operatorname{Ext}\nolimits}

\renewcommand{\Im}{\operatorname{Im}\nolimits}
\renewcommand{\ker}{\operatorname{Ker}\nolimits}

\newcommand{\bo}{\operatorname{b}\nolimits}

\newcommand{\RHom}{\mathbf{R}\strut\kern-.2em\operatorname{Hom}\nolimits}

\DeclareMathOperator{\add}{\mathsf{add}}

\DeclareMathOperator{\thick}{\mathsf{thick}}
\DeclareMathOperator{\CM}{\mathsf{CM}}

\DeclareMathOperator{\moduleCategory}{\mathsf{mod}} \renewcommand{\mod}{\moduleCategory}

\newcommand{\DDD}{\mathsf{D}}
\newcommand{\KKK}{\mathsf{K}}

\DeclareMathOperator{\vect}{\mathsf{vect}}
\DeclareMathOperator{\proj}{\mathsf{proj}}

\tikzset{every picture/.style={line width=0.75pt}} 

\begin{document}

	\title[A recollement approach to Brieskorn-Pham  singularities]{A recollement approach to Brieskorn-Pham  singularities}

	\author[W. Weng] {Weikang Weng}

	\makeatletter \@namedef{subjclassname@2020}{\textup{2020} Mathematics Subject Classification} \makeatother
	
	\subjclass[2020]{16G10, 16G50, 18G80}
	\keywords{Brieskorn-Pham singularity,  Cohen-Macaulay module,  recollement, tilting object, $m$-replicated algebra}
	
	\begin{abstract} 
		In this paper, we construct recollements and ladders for Brieskorn-Pham  singularities via  reduction/insertion functors,  and study the singularity categories of the Brieskorn-Pham  singularities  using these ladders. 
				
		In particular, we construct a class of tilting objects, called the extended tilting $n$-cuboids, whose endomorphism algebras are $n$-fold tensor products of certain Nakayama algebras.  Moreover, we show that   such an endomorphism algebra is derived equivalent to a certain replicated algebra.  This generalizes the Happel-Seidel symmetry to the context of Brieskorn-Pham singularities.
			
		
	\end{abstract}
	
	\maketitle
	
	\section{Introduction}
	
	The \emph{Brieskorn-Pham} (\emph{BP}) \emph{singularity} of a sequence $(p_1,\dots,p_n)$ with integers $p_i\ge 2$ over an algebraically closed field $\mathbf{k}$, is defined as  $$R:=\mathbf{k}[X_1, \dots, X_n]/(\sum_{i=1}^{n} X_i^{p_i}).$$ 	
	Brieskorn-Pham singularities and their (graded) singularity categories have been actively studied,  e.g. representation theory \cite{HIMO}, homological mirror symmetry \cite{FU} and automorphic forms \cite{Neumann:1977}.   In particular, for the case $n=3$, such a singularity is called a \emph{triangle singularity}. Its singularity category  has been investigated through multiple approaches, such as the stable category of vector bundles on weighted projective lines in the sense of Geigle-Lenzing \cite{KLM,KLM2,L1}, Cohen-Macaulay modules and matrix factorizations \cite{KST1,KST2}; for the case $n=4$, we refer to \cite{CRW}.
	
	Recollements of triangulated categories, first introduced by Beilinson, Bernstein, and Deligne in \cite{BBD}, arose in the study of derived categories of sheaves on topological spaces. The central idea is to investigate the `middle' large category by analyzing two smaller ones. Gluing techniques have been widely used in the literature to construct torsion pairs, tilting objects, silting objects, $t$-structures, co-$t$-structures and $n$-cluster tilting subcategories, for example, see \cite{AKL, Bo, DZ,LVY,LZZ,SZ}. 
	
	Ladders, later introduced by Beilinson, Ginzburg and Schechtman \cite{BGS}, generalize the notion of recollement. Roughly speaking, a ladder of triangulated categories is a collection of recollements of triangulated categories. The ladders have been used to study the derived simplicity of the derived module categories \cite{AKLD}, and the compactly generated triangulated categories \cite{BAS}. 
	We remind that a recollement of a triangulated category with Serre duality can be extended to a ladder by using the reflecting approach introduced by J{\o}gensen \cite{J}.	
	For the stable category of vector bundles over weighted projective lines, Chen \cite{C} constructed the reduction and insertion functors that yield a recollement, and later Ruan \cite{Ruan} extended this recollement to a ladder by interpreting these functors via 
	$p$-cycle constructions. 
	
	The aim of the paper is to generalize Chen's approach to construct a ladder for BP  singularities,  and study the singularity categories of the BP  singularities  using the ladder in terms of  Cohen-Macaulay modules. 
	
	
	Let $\L$ be the abelian group generated by the symbols $\x_1,\ldots,\x_n,\c$ modulo the relations $p_i\x_i=:\c$ for $1\le i \le n$. The algebra $R$ is $\L$-graded by setting $\deg X_i:=\x_i$, called an \emph{$\L$-graded Brieskorn-Pham} (\emph{BP}) \emph{singularity} with weights $p_1,\dots,p_n$. Denote by $\CM^{\L}R$ the category  of $\L$-graded (maximal) Cohen–Macaulay $R$-modules and by $\underline{\CM}^{\L} R$ its stable category. By results of Buchweitz \cite{Bu} and Orlov \cite{Orlov:2009}, there exists a triangle equivalence 
	$
		\DDD^{\L}_{\rm sg}(R)\simeq\underline{\CM}^{\L}R,
	$
	where $\DDD^{\L}_{\rm sg}(R):=\DDD^{\bo}(\mod^{\L}R)/\KKK^{\bo}(\proj^{\L}R)$
	is the \emph{singularity category} of $R$.
	
	We give a natural analog of  the reduction and insertion functors in \cite{C} for the BP singularities and extend them by using the reflecting approach \cite{J}.  This yields the following result, which covers \cite[Theorem 5.2]{C} and \cite[Theorem 4.9]{Ruan}.

	\begin{theorem}	[Theorem \ref{ladder diagram}]\label{Thm A} For $j=1,2$, let $R^{j}$ be an $\L_j$-graded BP singularity  with weights $p_1,\ldots,p_{n-1},p_{j,n}$ such that $p_{1,n}+p_{2,n}=p_n+1$. Write $q:=p_{1,n}$. Then the following diagram is an infinite ladder of period $p_n$.
		\begin{align*}
			\xymatrix@C=1cm{
				\underline{\CM}^{\L_1}{R^1}
				\ar@/_3pc/[rrr]_{\vdots}|{{\psi}_{1,q}}
				\ar@/^3pc/[rrr]^{\vdots}|{{\psi}_{1,q-2}}
				\ar[rrr]|{{\psi}_{1,q-1}}
				&&& \underline{\CM}^{\L}{R}
				\ar@/_1.5pc/[lll]|{{\phi}_{1,q-1} }
				\ar@/^1.5pc/[lll]|{{\phi}_{1,q}}
				\ar[rrr]|{{\phi}_{2,0}}
				\ar@/_3pc/[rrr]_{\vdots}|{{\phi}_{2,1}}
				\ar@/^3pc/[rrr]^{\vdots}|{{\phi}_{2,-1}}
				&&& \underline{\CM}^{\L_2}{R^2}
				\ar@/_1.5pc/[lll]|{{\psi}_{2,-1}}
				\ar@/^1.5pc/[lll]|{{\psi}_{2,0}}
			}
		\end{align*}
		Here, the functors in the ladder are explicitly constructed by (\ref{Reduction/insertion functor}).		
	\end{theorem}

	Due to the explicit formulas for the action of the reduction/insertion functors and the suspension functor on a  class of Cohen-Macaulay modules, see Lemma \ref{L-action to U} and Proposition \ref{c=[2]}, we construct a class of tilting objects in $\underline{\CM}^{\L} R$ using the ladder. Let  $\vdelta :=\sum_{i=1}^{n}(p_i-2)\x_i$ and $\s:=\sum_{i=1}^{n}\x_i$. For each  $\vell=\sum_{i=1}^{n}\ell_i\x_i\in[\s,\s+\vdelta]$, let
	$$E^{\vell}:=R/(X_i^{\ell_i}\mid 1\le i\le n)\in \mod^{\L} R ~\text{ and } ~U^{\vell}:=\rho(E^{\vell}),$$ 
	where $\rho$ denotes the composition $\mod^{\L}R \subset \DDD^{\bo}(\mod^{\L}R)\to\DDD_{\rm sg}^{\L}(R)\xrightarrow{\sim} \underline{\CM}^{\L}R$.	 	
	For each $\x=\sum_{i=1}^{n} \lambda_i\x_i$ with $0\le \x\le \vdelta$, let $ \sigma(\x):= \sum_{i=1}^{n} \lambda_i$. 
	Denote by $[1]$  the suspension functor in the triangulated category $\underline{\CM}^{\L} R$. Denote by $\mathbf{k}\vec{\AA}_{n}(m):=\mathbf{k}\vec{\AA}_n/{\rm rad^m} \, \mathbf{k}\vec{\AA}_n$ for  $n,m \ge 1$. Here $\mathbf{k}\vec{\AA}_n$ denotes the path algebra of the equioriented quiver of type $\AA_n$ and ${\rm rad}\, \mathbf{k}\vec{\AA}_n$ denotes the ideal generated by all arrows of $\mathbf{k}\vec{\AA}_n$.
	
	\begin{theorem} [Theorem \ref{tilting object}] \label{Thm B}  For a subset  $I$ of $\{1,\dots,n\}$, we set $\vdelta_I:=\sum_{i\in I}(p_i-2)\x_i$ and $I^{c}:=\{1,\dots,n\}\setminus I$.  Then the object
		$$V=V_I:=\bigoplus_{\s\leq\vell\leq\s+\vdelta_I}\bigoplus_{0\leq\x\leq\vdelta-\vdelta_{I}}U^{\vell}(\x)[-\sigma(\x)]$$	
		is   tilting  in $\underline{\CM}^{\L} R$, called the \emph{extended tilting $n$-cuboid} with endomorphism algebra $$\underline{\End}(V)^{\rm op}\simeq \, \bigotimes_{ i\in  I} \mathbf{k}\vec{\AA}_{p_i-1} \otimes  \bigotimes_{i\in I^{c}} \mathbf{k}\vec{\AA}_{p_i-1}(2). $$ 
		Moreover, the objects $U^{\vell}(\x)[-\sigma(\x)]$ $(\s\leq\vell\leq\s+\vdelta_I,\, 0\le \x \le \vdelta-\vdelta_I)$ can be arranged to form a full, strong exceptional sequence.	
	\end{theorem}
	
	As one extreme case $I=\{1,\dots,n\}$ of Theorem \ref{Thm B},  $\bigoplus_{\s\leq\vell\leq\s+\vdelta}U^{\vell}$  is tilting  in $\underline{\CM}^{\L} R$, called the \emph{tilting $n$-cuboid}, whose endomorphism algebra is $\bigotimes_{i=1}^{n} \mathbf{k}\vec{\AA}_{p_i-1}$. 
	This is known by Kussin-Lenzing-Meltzer \cite{KLM} for the case $n=3$, by Futaki-Ueda \cite{FU} and Herschend-Iyama-Minamoto-Oppermann \cite{HIMO} for general $n$. 

	As a refinement of Proposition \ref{thick}(a),  $\rho(\mathbf{k})(\x)$ $(0\le \x \le \vdelta)$ can be arranged to form a full exceptional sequence but not strong in general. Thus it is natural to pose the following question, which is a higher version of the question posed by Kussin-Lenzing-Meltzer \cite[Remark 6.10]{KLM} in terms of Cohen-Macaulay modules.
	
	\begin{question} [Question \ref{question}] \label{question 0}
		Does there exist a full, strong exceptional sequence in $\underline{\CM}^{\L} R$ from the objects $\rho(\mathbf{k})(\x)$, $\x\in\L$.
	\end{question}

	There are some partial answers of Question \ref{question 0}, which are given by \cite{KLM,KLM2} for the type $(2,3,n)$ with $n \ge 2$, and by \cite{DR} for the type $(2,a,b)$ with $a,b\ge2$.	 				
	As the other extreme case $I=\emptyset$ of Theorem \ref{Thm B}, $\bigoplus_{0\leq\x\leq\vdelta}\rho(\mathbf{k})(\x)[-\sigma(\x)]$ can be arranged to form a full, strong exceptional sequence in $\underline{\CM}^{\L} R$.  
	This yields a positive answer to Question \ref{question 0} for type $(2,p_2,\dots,p_n)$ with $p_i\ge 2$. 
	
	Happel-Seidel symmetry, established in \cite{HS},  concerns the derived equivalence of homogeneous acyclic Nakayama algebras.
	A general and simple explanation of  this symmetry is provided in \cite{KLM} through the stable categories of vector bundles over  weighted projective lines. More precisely, it is shown that the Nakayama algebras  $\mathbf{k}\vec{\AA}_{m}(p_1)$, $\mathbf{k}\vec{\AA}_{m}(p_2)$, and the incidence algebra $\mathbf{k}\vec{\AA}_{p_1-1}\otimes \mathbf{k}\vec{\AA}_{p_2-1}$ are derived equivalent, where $m:=(p_1-1)(p_2-1)$. 
	In our treatment, we show that the $(p_t-2)$-replicated algebras of $\bigotimes_{i\in S\setminus\{t\}}\mathbf{k}\vec{\AA}_{p_i-1}$ for $t\in S:=\{1\dots,n\}$, are realized as endomorphism algebras of tilting objects on $\underline{\CM}^{\L} R$. This following result   extends the Happel-Seidel symmetry to the setting of BP singularities for general $n$.

	\begin{theorem}[Corollaries \ref{nHappel-Seidel Symmetry}, \ref{Happel-Seidel Symmetry n=2}] \label{Thm C} For each $t\in S:=\{1\dots,n\}$, the $(p_t-2)$-replicated algebra of $\bigotimes_{i\in S\setminus\{t\}}\mathbf{k}\vec{\AA}_{p_i-1}$ and the stable category $\underline{\CM}^{\L} R$ are derived equivalent. 
		In particular, when $n=2$, the Nakayama algebras $\vec{\AA}_{m}(p_1)$ and $\vec{\AA}_{m}(p_2)$ are derived equivalent, where $m:=(p_1-1)(p_2-1)$.
		
	\end{theorem}
	

	The paper is organized as follows. In Section \ref{sec: Recollements and Brieskorn-Pham singularities}, we recall some basic concepts and properties on recollements, ladders and  BP singularities. In Section \ref{sec: A ladder for Brieskorn-Pham  singularities}, we introduce the reduction and insertion functors, and provide an explicit formula for the action of these functors on a distinguished class of Cohen-Macaulay modules, see Lemma \ref{L-action to U}. This yields a ladder for BP  singularities, as established in Theorem \ref{Thm A}. In Section \ref{sec: Eisenbud periodicity}, we study the action of the suspension functor on  the class of Cohen-Macaulay modules, see Proposition \ref{c=[2]}, leading to a version of Eisenbud periodicity. In Section \ref{sec: Tilting objects via tensor products}, we construct a class of tilting objects via the ladder, see Theorem \ref{Thm B}. Two particularly important cases are treated in Corollaries \ref{tilting n-cub} and \ref{tilting n-cub  2}. In Section \ref{sec: Happel-Seidel symmetry}, we  generalize  the Happel-Seidel symmetry to the context of BP singularities for general $n$, as stated in Theorem \ref{Thm C}. In Section \ref{sec: Examples}, we  illustrate our results through concrete examples and  establish some derived equivalences for the replicated algebras of the Dynkin type.

	\section{Recollements and Brieskorn-Pham singularities} \label{sec: Recollements and Brieskorn-Pham singularities}
	In this section, we  recall some  basic concepts and facts about recollements, ladders and Cohen-Macaulay representations on Brieskorn-Pham singularities in \cite{HIMO}. Throughout this paper we fix an algebraically closed field $\mathbf{k}$, and denote  by   $D(-):=\Hom_{\mathbf{k}}(-,\mathbf{k})$. All modules are left modules.
	
		\subsection{Recollements and ladders} 	Recall the concepts of the recollements and ladders.
		Let $\mathcal{T}$ and $\mathcal{T'}$ be two arbitrary categories, and 
		\begin{eqnarray}\label{adjoint pair diagram} 
			\xymatrix@!=3pc{\mathcal{T}
				\ar@<1ex>[r]^{F} &\mathcal{T'}
				\ar@<1ex>[l]^{G}
			}
		\end{eqnarray} be a pair of functors between them. The pair $(F,G)$ is called \emph{adjoint pair} if  there exists a functorial isomorphism $$\Hom_{\mathcal{T'}}(FA,B)\simeq \Hom_{\mathcal{T}}(A,GB)$$ for any $A\in \mathcal{T}$ and $B\in \mathcal{T'}$. A sequence of functors $(F,G,H)$ is called an \emph{adjoint triple} if both $(F,G)$ and $(G,H)$ are adjoint pairs.
		
		It is well known that adjoint pairs are compatible with auto-equivalences of categories and also with compositions. More precisely, if $(F,G)$ is an adjoint pair between $\mathcal{T}$ and $\mathcal{T'}$, and there are auto-equivalences $\eta$ of $\mathcal{T}$ and $\eta'$ of $\mathcal{T'}$ respectively, then $(\eta' F\eta, \eta^{-1}G (\eta')^{-1})$ is also an adjoint pair.
		
		\begin{definition} \cite{BBD}
			Let $\mathcal{T}, \mathcal{T}_1, \mathcal{T}_2$ be triangulated categories. A \emph{recollement} of $\mathcal{T}$ relative to $\mathcal{T}_1$ and $\mathcal{T}_2$, denoted by $(\mathcal{T}_1, \mathcal{T}, \mathcal{T}_2)$, is a diagram
			\begin{equation}\label{recollement}
				\xymatrix{
					\mathcal{T}_1\ar[rrr]|{i_{\ast}}
					&&& \mathcal{T}
					\ar@/_1.5pc/[lll]|{i^{\ast}}
					\ar@/^1.5pc/[lll]|{i^!}
					\ar[rrr]|{j^{\ast}}
					&&& \mathcal{T}_2
					\ar@/_1.5pc/[lll]|{j_!}
					\ar@/^1.5pc/[lll]|{j_{\ast}}
				}
			\end{equation}
			of triangulated categories and triangle functors such that
			\begin{itemize}
				\item[(R1)] $(i^{\ast},i_{\ast},i^{!})$ and $(j_!,j^{\ast},j_{*})$ are adjoint triples;
				\item[(R2)]  $i_\ast,\,j_\ast,\,j_!$  are fully faithful;
				\item[(R3)]  $\Im (i_\ast)=\ker (j^{\ast})$. 
			\end{itemize}	
		\end{definition}
		Note that under the conditions (R1) and (R2), the condition (R3) is equivalent to  $\Im (j_!)=\ker (i^{\ast})$ and also to  $\Im (j_{\ast})=\ker (i^!)$. We say that two recollements of $\mathcal{T}$ relative to $\mathcal{T}_1$ and $\mathcal{T}_2$ are \emph{equivalent} if the TTF triples associated to them coincide, that is, the related subcategories $(\Im (j_!),\ker (j^{\ast}), \Im (j_{\ast}))$ coincide.
		
		A \emph{ladder} $\mathcal{L}$ is a finite or infinite diagram of  triangulated categories and triangle functors
		\begin{equation}\label{ladder}
			\xymatrix@C=1cm{
				\mathcal{T}_1 \ar @/^-3pc/[rrr]|{\psi_{1}}_{\vdots}
				\ar @/^3pc/[rrr]_{}|{\psi_{-1}}^{\vdots}
				\ar[rrr]|{\psi_0}
				&&& \mathcal{T}
				\ar@/_1.5pc/[lll]|{\psi^{0}}
				\ar@/^1.5pc/[lll]|{\psi^{1}}
				\ar[rrr]|{\phi^{0}}
				\ar @/^-3pc/[rrr]|{\phi^{1}}_{\vdots}
				\ar @/^3pc/[rrr]_{}|{\phi^{-1}}^{\vdots}
				&&& \mathcal{T}_2
				\ar@/_1.5pc/[lll]|{\phi_{-1}}
				\ar@/^1.5pc/[lll]|{\phi_{0}}
			}
		\end{equation}
		such that any three consecutive rows form a recollement. Multiple occurrence of the same recollement is allowed. The \emph{height} of a ladder is the number of recollements contained in it (counted with multiplicities). A recollement is considered to be a ladder of height one. 	
		A ladder is called \emph{periodic} if there exists a positive integer $n$ such that the $n$-th recollement going upwards (respectively, going downwards) in the ladder is equivalent to the recollement which is considered to be a ladder of height one. The minimal such positive integer $n$ is the \emph{period} of the ladder. If the middle term of a recollement of triangulated categories has Serre duality, then by using the reflecting approach induced by J{\o}gensen \cite{J} we can obtain a ladder.

		\subsection{Brieskorn-Pham singularities} Let  an integer $n\ge 1$. 
		Fix an $n$-tuple $\mathbf{p}=(p_1,\dots,p_n)$ of each integer $p_i \ge 2$, called \emph{weights}. Consider the $\mathbf{k}$-algebra
		\begin{align} \label{GL hypersurface singularities}
			R=R(\mathbf{p}):=\mathbf{k}[X_1,\dots,X_n] / (X_1^{p_1}+ \cdots +X_n^{p_n}).
		\end{align} 
		Let $\L=\L(\mathbf{p})$ be the abelian group on generators $\x_1,\ldots,\x_n,\c$ modulo relations $$p_1 \x_1= \dots=p_n \x_n=:\c.$$ The element $\c$ is called the \emph{canonical element} of $\L$. Each element $\x$ in $\L$ can be uniquely written in a \emph{normal form} as 
		\begin{align}\label{equ:nor}
			\x=\sum_{i=1}^n\lambda_i\x_i+\lambda\c
		\end{align}
		with $ 0\le \lambda_i < p_i$ and $\lambda\in\Z$. We will write $l_i(\x):=\lambda_i$ if necessary.  We can regard $R$ as an $\L$-graded $\mathbf{k}$-algebra by setting $\deg X_i:=\x_i$   for any $i$, hence $R=\bigoplus_{\x\in \L } R_{\x}$, where $R_{\x}$ denotes the homogeneous component of degree $\vec{x}$. 
		The algebra $R$ is called an \emph{$\L$-graded Brieskorn-Pham} (\emph{BP}) \emph{singularity} associated with weights $p_1,\dots,p_n$ in \cite{FU}. The pair $(R,\L)$ is called a \emph{Geigle-Lenzing} (\emph{GL}) \emph{hypersurface singularity} of dimension $n-1$, as introduced in \cite{HIMO}.
		
		Let $\L_+$ be the submonoid of $\L$ generated by $\x_1,\ldots,\x_n,\c$. Then we equip $\L$ with the structure of a partially ordered set: $\x\le \y$ if and only if $\y-\x\in\L_+$. Each element $\x$ of $\L$ satisfies exactly one of the following two possibilities
		\begin{align}\label{two possibilities of x}
			\x\geq 0 \text{\quad or\quad}  \x\leq (n-2)\c+\vec{\omega},
		\end{align}
		where  $\w:=\c-\sum_{i=1}^n \x_i$ is called the \emph{dualizing element} of $\L$. 
		Note that $R_{\x} \neq 0$ if and only if $\x \geq 0$ if and only if $\lambda \geq 0$ in its normal form in (\ref{equ:nor}).  
		
		Denote by  $C$ the \emph{$(\Z\c)$-Veronese subalgebra} of $R$, that is, $C=\bigoplus_{a\in \Z}R_{a\c}$. Clearly, $$C=\mathbf{k}[X_1^{p_1},\dots,X_{n-1}^{p_{n-1}}]$$ is the $(\Z\c)$-graded polynomial algebra in the variables $X_1^{p_1},\dots, X_{n-1}^{p_{n-1}}$.		 Let $\x=\sum_{i=1}^n\lambda_i\x_i+\lambda\c\in \L$ be in normal form. Each homogeneous component $R_{\x}$ has an explicit basis $\{\prod_{i=1}^{n-1} X_i^{\lambda_i+d_ip_i}X_n^{\lambda_n} \mid \sum_{i=1}^{n-1} d_i=\lambda,~ d_i\ge 0 \}$. Moreover, $R$ is a free $C$-module of rank $p_1p_2\dots p_n$ with a basis $\{ X_1^{a_1}\dots X_n^{a_n}\mid 0 \le a_i <p_i\}$, see \cite[Proposition 3.5]{HIMO} for details. 
		
		Denote by $\mod^{\L} R$ the category of finitely generated $\L$-graded $R$-modules. For $M\in \mod^{\L} R$, we write $M=\bigoplus_{\x\in \L } M_{\x}$, where $M_{\x}$ is the homogeneous component of degree $\vec{x}$. For any $\y \in \L$, the \emph{shift module} $M(\y)$ is the same as $M$ as a ungraded module, whereas it is graded such that  $(M(\y))_{\x}=M_{\x+\y}$.	This yields the \emph{degree shift functor} $(\y):\mod^{\L} R\to \mod^{\L} R$, which is clearly an automorphism of $\mod^{\L} R$. 
						
		Note that the subcategory $\proj^{\L} R$ of $\mod^{\L} R$ consisting of all $\L$-graded projective  $R$-modules, is given by $\add \{R(\y)\mid \y\in \L\}$, that is, the subcategory consisting of direct summands of finite direct sums of modules in the set $\{R(\y)\mid \y\in \L\}$.	Note that the algebra $R$ has a unique $\L$-graded maximal ideal $R_+:=\bigoplus_{\x\in \L_{+}}R_{\x}$. Then the set  of all  $\L$-graded simple $R$-modules is  given by $\{\mathbf{k}(\y)\mid \y\in \L\}$.
 		
		\subsection{Cohen-Macaulay representations} 	
		Recall that a module $M\in\mod^{\L} R$ is called $\L$-graded \emph{(maximal) Cohen-Macaulay} if $\Ext_{\mod^{\L}R}^{i}(M,R)=0$ holds for all $i\ge 1$. Denote by $\CM^{\L} R$ the category of $\L$-graded  Cohen-Macaulay $R$-modules.
		 							
		Consider the algebra $C$ now as an $\L$-graded subalgebra of $R$. By the embedding homomorphism $C \to R$ of $\L$-graded algebras, an $\L$-graded $R$-module $M$ induces $\L$-graded $C$-modules $M|_{\x+\Z\c}$ for all $\x\in \L$. Here, $M|_{\x+\Z\c}=\bigoplus_{a\in \Z}M_{\x+a\c}$. The following observation gives an analog of \cite[Lemma 3.3]{C} for BP singularities.
		%
		
			\begin{lemma} \label{CM}  Under the above setting, $M\in \mod^{\L} R$ belongs to $\CM^{\L} R$ if and only if  all the induced $C$-modules $M|_{\x+\Z\c}$ $(\x\in \L)$ are finite generated projective.		
		\end{lemma}
		\begin{proof} 
			Clearly, $M$ is finite generated if and only if all the induced $C$-modules $M|_{\x+\Z\c}$ $(\x\in \L)$ are finite generated.		
			
			Note that the algebra $R$ is a graded complete intersection of Krull dimension $n-1$ and then $R$ is graded Gorenstein of self-injective dimension $n-1$. By a graded version of local duality,  $M\in \CM^{\L} R$ if and only if $\Ext^{i}_{\mod^{\L} R}(\mathbf{k}(\y),M)=0$ for any $i\le n-2$ and $\y\in\L$. 
			This is equivalent to that $\Ext^{i}_{\mod^{\L} C}(\mathbf{k}(\y),M)=0$ for any $i\le n-2$ and $\y\in\L$
			by using a similar argument as in the proof of the third paragraph of \cite[Theorem 5.1]{GL}. Again by the local duality, this is equivalent to $M\in \CM^{\L} C$. Since the algebra $C$ is  regular which has graded global dimension $n-1$, 
			$M\in \CM^{\L} C$ if and only if $M\in \proj^{\L} C$, which is further equivalent to all the induced $C$-modules $M|_{\x+\Z\c}$ $(\x\in \L)$ are projective.		
		\end{proof}						
		 Since $R$ is Gorenstein, the category $\CM^{\L} R$  
		is Frobenius, and thus by a general result of Happel \cite{Hap1}, its stable category $\underline{\CM}^{\L}R$ forms a triangulated category. By results of Buchweitz \cite{Bu} and Orlov \cite{Orlov:2009}, there exists a triangle equivalence 
		\begin{align}\label{CM equ.}
			\DDD^{\L}_{\rm sg}(R)\simeq\underline{\CM}^{\L}R,
		\end{align}
		 where $\DDD^{\L}_{\rm sg}(R):=\DDD^{\bo}(\mod^{\L}R)/\KKK^{\bo}(\proj^{\L}R)$
		is the \emph{singularity category} of $R$.

		Consider the composition $$\rho: ~\mod^{\L}R\subset\DDD^{\bo}(\mod^{\L}R)\to\DDD_{\rm sg}^{\L}(R)\xrightarrow{\sim} \underline{\CM}^{\L}R,$$
		where the first inclusion identifies a module as a stalk complex concentrated at degree zero, and the second functor denotes the quotient functor, and the last equivalence  is given by (\ref{CM equ.}).
		
		The following results are well-known.
		
		\begin{proposition}\label{thick}
			\begin{itemize}
				\item[(a)]  $\underline{\CM}^{\L} R=\thick \langle \rho{(\mathbf{k})}(\y) \mid \y\in \L \rangle$.
				\item[(b)] \emph{(Kn\"orrer periodicity)} Let $(R',\L')$ be an $\L'$-graded  BP  singularity with weights $2,p_1,\dots,p_n$. Then there is a triangle equivalence $\underline{\CM}^{\L} R \simeq \underline{\CM}^{\L'} R'$.
			\end{itemize}			
		\end{proposition}
		\begin{proof} (a) By \cite[Theorem 3.32]{HIMO}, $R$ has $\L$-isolated singularities. Thus
		the claim follows from a graded version of \cite[Proposition A.2]{KMV}; also see \cite[Theorem 4.10]{HIMO}.
			
			(b) This is well-known, we refer to \cite{Kn}; also see \cite[Corollary 4.23]{HIMO}.
		\end{proof}
		
		\begin{theorem}\cite[Theorem 4.3]{HIMO} \label{Auslander-Reiten-Serre duality} (Auslander-Reiten-Serre duality) For any $X,Y\in \underline{\CM}^{\L}R$, there exists an isomorphism 
			\begin{equation*}\label{Auslander-Reiten-Serre duality CM}
				\Hom_{\underline{\CM}^{\L}R}(X,Y)\simeq D\Hom_{\underline{\CM}^{\L}R}(Y,X(\w)[n-2]).
			\end{equation*}	
			In other word, $\underline{\CM}^{\L} R$ has a Serre functor $(\w)[n-2]$.
		\end{theorem}

		A simple but important example is the following.
		
		\begin{example} Let $n=1$ and assume that $(R,\L)=(\mathbf{k}[X_1]/(X_1^{p_1}),\langle\x_1\rangle)$ with $p_1\ge 2$. Then  ${\CM}^{\L} R\simeq \mod^{\L} R$ and thus $\underline{\CM}^{\L} R\simeq \underline{\mod}^{\L} R\simeq \DDD^{\bo}(\mod \mathbf{k}\vec{\AA}_{p_1-1})$. Here $\vec{\AA}_{p_1-1}$ is the path algebra of the following quiver: 
			\[
			\begin{xy} 0;<4pt,0pt>:<0pt,4pt>:: 
				(0,0) *+{1} ="0",
				(10,0) *+{2} ="1",
				(20,0) *+{\dots} ="11",
				(32,0) *+{p_1-2} ="111",
				(44,0) *+{p_1-1} ="1111",
				"0", {\ar"1"},
				"1", {\ar"11"},
				"11", {\ar"111"},
				"111", {\ar"1111"},
			\end{xy}
			\]	
		\end{example}
		
		For convenient, we denote ${\Hom}_{\underline{\CM}^{\L}R}(-,-)$ by  $\underline{\Hom}(-,-)$.	
		
%

	\section{A ladder for Brieskorn-Pham  singularities}\label{sec: A ladder for Brieskorn-Pham  singularities}
	In this section, we introduce the reduction and insertion functors, and
	investigate the action of these functors on certain Cohen-Macaulay modules.
	Further, we give an explicit construction of a ladder for Brieskorn-Pham  singularities.
	The treatment of this section is parallel to the exposition in \cite{C}.

	Throughout this section, we let $R$ be an $\L$-graded BP singularity with weights $p_1,\ldots,p_{n-1}, p_n$. For $j=1,2$, let $R^{j}$ be an $\L_j$-graded BP singularity  with weights $p_1,\ldots,p_{n-1},p_{j,n}$, where $p_{1,n}+p_{2,n}=p_n+1$. Thus
	\begin{eqnarray*} 
		&	R^{j}=\mathbf{k}[X_1,\dots,X_n] / (X_1^{p_1}+ \cdots+X_{n-1}^{p_{n-1}} +X_n^{p_{j,n}}),&\\
		&\L_j=\langle\c_j,\x_{j,1},\ldots,\x_{j,n}\rangle/\langle p_{j,n}\x_{j,n}-\c_j,~p_{i}\x_{j,i}-\c_j\mid1\le i\le n-1\rangle.&
	\end{eqnarray*}
	Consider the injective map $$\theta_j:\L_{j}\to \L$$ which sends an element $\x=\sum_{i=1}^n\lambda_i\x_{j,i}+\lambda\c_j$ to $\theta_j(\x)=\sum_{i=1}^n\lambda_i\x_i+\lambda\c$. Here $\x\in\L_j$ is in its normal form, that is, $0 \le \lambda_i < p_i$ with $1\le i <n$, $0\le \lambda_n < p_{j,n}$ and $\lambda \in \Z$. Note that  the map $\theta_j$ is generally not  a homomorphism of groups. We also observe that $\x\in\L$ belongs to the image of $\theta_j$ if and only if $\lambda_n < p_{j,n}$.

		\subsection{Reduction functor $\phi_{j,k}$ and insertion functor $\psi_{j,k}$} \label{RI functors}	
	For each $j=1,2$, a pair of functors $(\phi_{j,0},\psi_{j,0})$ is defined as follows.
	The functor $\phi_{j,0}: \mod^{\L} R \to \mod^{\L_j} R^j$ is given by $\phi_{j,0}(M)=\bigoplus_{\x\in \L_j} \phi_{j,0}(M)_{\x}$, where 
	$$ \phi_{j,0}(M)_{\x}:= M_{\theta_j(\x)+(p_n-p_{j,n})\x_n}.$$
	Conversely,   $\psi_{j,0}: \mod^{\L_j} R^j \to \mod^{\L} R$ is given by $\psi_{j,0}(N)=\bigoplus_{\x\in \L} \psi_{j,0}(N)_{\x}$, where 
	\[\psi_{j,0}(N)_{\x}:=\left\{\begin{array}{ll}
		N_{\theta_j^{-1}(\x-\lambda_n \x_n)} & \text{if }~0\le \lambda_n< p_n-p_{j,n},\\
		N_{\theta_j^{-1}(\x-(p_n-p_{j,n})\x_n)} &\text{otherwise}.
	\end{array}\right.\]				
	
	We point out that both the functors $\phi_{j,0}$ and $\psi_{j,0}$ are analogs of the functors considered in \cite[Section 4]{C}, see also
	\cite[Section 9]{GL2}. Using the reflecting approach introduced by J{\o}gensen \cite{J}, $\phi_{j,0}$ and $\psi_{j,0}$ can be extended  
	 as follows: for each $j=1,2$ and $k\in \Z$, we denote by
	 \begin{equation}\label{Reduction/insertion functor}
	 	\phi_{j,k}:=(-k\x_{j,n})\phi_{j,0}(k\x_n)\ \text{ and }\ \psi_{j,k}:=(-k\x_n)\psi_{j,0}(k\x_{j,n}).
	 \end{equation}
	We call $\phi_{j,k}$ (resp. $\psi_{j,k}$) \emph{reduction functor} (resp. \emph{insertion functor}).
	
	
	Note that the action of $\c_j$ and $\x_{j,i}$ ($1\le i<n$) on $\phi_{j,0}(M)$ is induced by the image of the one on $M$, that is, $\phi_{j,0}(M)(\c_j)=\phi_{j,0}(M(\c))$ and $\phi_{j,0}(M)(\x_{j,i})=\phi_{j,0}(M(\x_i))$, while the action of  $\x_{j,n}$ on $\phi_{j,0}(M)$ is given by
	\[\big(\phi_{j,0}(M)(\x_{j,n})\big)_{\x}=\left\{\begin{array}{ll}
		M_{\theta_j(\x)+(2p_n-2p_{j,n}+1)\x_n} & \text{if }~\lambda_n=p_{j,n}-1 ,\\
		M_{\theta_j(\x)+(p_n-p_{j,n}+1)\x_n} &\text{otherwise}.
	\end{array}\right.\]	
	The action of $\c$ and $\x_{i}$ ($1\le i<n$) on $\psi_{j,0}(N)$ is induced by the preimage of the one on $N$, that is, $\psi_{j,0}(N)(\c)=\psi_{j,0}(N(\c_j))$ and $\psi_{j,0}(N)(\x_{i})=\psi_{j,0}(N(\x_{j,i}))$, while the action of  $\x_{n}$ on $\psi_{j,0}(N)$ is given by
	\[(\psi_{j,0}(N)(\x_n))_{\x}=\left\{\begin{array}{ll}
		N_{\theta_j^{-1}(\x-\lambda_n \x_n)} &\text{if }~0\le \lambda_n< p_n-p_{j,n},\\
		N_{\theta_j^{-1}(\x-(p-p_{j,n}-1)\x_n)} &\text{otherwise}.
	\end{array}\right.\]

	\begin{lemma} \label{RI-adjoint triple} For any $j=1,2$ and $k\in \Z$,  the following assertions hold.
		\begin{itemize}
			\item[(a)]  The functor $\psi_{j,k}$ is fully faithful. 
			\item[(b)] 	$(\phi_{j,k},\psi_{j,k},\phi_{j,k+1})$ forms an adjoint triple.
			\item[(c)]  Both $\phi_{j,k}$ and $\psi_{j,k}$ are exact functors. 
			\item[(d)]  Both $\phi_{j,k}$ and $\psi_{j,k}$ preserve Cohen-Macaulay modules.
		\end{itemize}		
	\end{lemma}
	\begin{proof} It suffices to show that the assertion holds for $k=0$. The proof is parallel to \cite[Lemmas 4.1--4.3]{C}. For the convenience of the reader, we include the proof.
		
		(a) It is obvious from the construction of $\psi_{j,0}$.
		
		(c) This is an immediate consequence of (b).
		
		(b) Let $M\in \mod^{\L} R$ and $N\in \mod^{\L_j} R^{j}$. We first give the construction of the isomorphism $\alpha:\Hom_{\mod^{\L_j} R^j}(\phi_{j,0}(M),N)\simeq\Hom_{\mod^{\L} R}(M,\psi_{j,0}(N))$.  It sends $f:\phi_{j,0}(M)\to N$ to $\alpha(f):M\to \psi_{j,0}(N)$ such that 
		$M_{\x}\to (\psi_{j,0}(N))_{\x}$ ($\x\in \L$) is given by the restriction of $f$ for the case $p_n-p_{j,n}\le \lambda_n<p_n$, and it is given by the composition $M_{\x}\xrightarrow{X_n^{p_n-p_{j,n}-\lambda_n}}M_{\x+(p_n-p_{j,n}-\lambda_n)\x_n}=\phi_{j,0}(M)_{\theta_j^{-1}(\x-\lambda_n \x_n)}\xrightarrow{f} N_{\theta_j^{-1}(\x-\lambda_n \x_n)}=(\psi_{j,0}(N))_{\x}$ otherwise. 
		 		
		On the other hand, we describe the isomorphism $\beta:\Hom_{\mod^{\L} R}(\psi_{j,0}(N),M)\simeq\Hom_{\mod^{\L_j} R^j}(N,\phi_{j,1}(M))$. It sends $g:\psi_{j,0}(N)\to M$ to $\beta(g):N\to \phi_{j,1}(M)$  such that $N_{\x}\to \phi_{j,1}(M)_{\x}$ ($\x\in\L_j$) is given by $N_{\x}=(\psi_{j,0}(N))_{\theta_j(\x)}\xrightarrow{g}M_{\theta_j(\x)}=\phi_{j,1}(M)_{\x}$ for the case $\lambda_n=0$, and
		it is given by $N_{\x}=(\psi_{j,0}(N))_{\theta_j(\x)+(p_n-p_{j,n})\x_n}\xrightarrow{g} M_{\theta_j(\x)+(p_n-p_{j,n})\x_n}=\phi_{j,1}(M)_{\x}$ otherwise.
		
		(d) 
		Let $M\in \CM^{\L} R$ and $\x\in \L_j$. Recall that $C=\mathbf{k}[X_1^{p_1},\dots,X_{n-1}^{p_{n-1}}]$. Then we have
		$\phi_{j,0}(M)|_{\x+\Z\c_j}=M_{\theta_j(\x)+(p_n-p_{j,n})\x_n+\Z\c}$ as $C$-module. By Lemma \ref{CM}, we have that $\phi_{j,0}$ preserves Cohen-Macaulay modules.				
		On the other hand, let $N\in \CM^{\L_j} R^j$ and $\z=\sum_{i=1}^{n}z_i\x_i+z\c \in \L$ be in normal form. Then as $C$-module 
		$\psi_{j,0}(N)|_{\z+\Z\c}=N_{\theta_j^{-1}(\z-z_n \x_n)+\Z\c}$  if $0\le z_n<p_n-p_{j,n}$, and $N_{\theta_j^{-1}(\z-(p_n-p_{j,n})\x_n)+\Z\c} $ otherwise.  By Lemma \ref{CM}, we have that $\psi_{j,0}$ preserves Cohen-Macaulay modules.
	\end{proof}
	
	The following observation induces that both the functors $\phi_{j,k}$ and $\psi_{j,k}$ preserve projective modules.
	\begin{lemma} \label{preserve projective modules} For any $j=1,2$ and $k\in \Z$,  the following assertions hold.
		\begin{itemize}
			\item[(a)] Let $\y=\sum_{i=1}^{n}y_i\x_{i}+y\c\in \L$ be in normal form. Then 
			\[\phi_{j,k}(R(\y))= \left\{\begin{array}{ll}
				R^j (\theta_j^{-1}(\y-(bp_n-k)\x_n)+(bp_{j,n}-k)\x_{j,n})	&\text{if }~ 0\le a < p_{j,n},\\
				R^j	(\theta_j^{-1}(\y-y_n\x_n)+((b+1)p_{j,n}-k)\x_{j,n})&\text{otherwise}. 
			\end{array}\right.\]
			Here $(y_n+k)\x_n=a\x_n+b\c$ be written in normal form in $\L$.
			\item[(b)]	Let $\y=\sum_{i=1}^{n}y_i\x_{j,i}+y\c_j\in \L_j$ be in normal form. Then $$\psi_{j,k}(R^{j}(\y))= R(\theta_j(\y-(bp_{j,n}-k)\x_{j,n})+(bp_n-k)\x_n).$$ 	
			Here $(y_n+k)\x_{j,n}=a\x_{j,n}+b\c_j$ be written in normal form in $\L_j$.
		\end{itemize}
	\end{lemma}
	
	\begin{proof} We only prove  (a), since (b) can be shown similarly.
		
		(a) First we show the assertion for $k = 0$, that is, $\phi_{j,0}(R(\y))\simeq  R^j(\theta_j^{-1}(\y))$ if $0\le y_n < p_{j,n}$, and $R^j(\theta_j^{-1}(\y-y_n\x_n)+\c_j)$ otherwise. Let $\x=\sum_{i=1}^{n}\lambda_i\x_{j,i}+\lambda\c_j\in \L_j$ be in normal form. We consider the following two cases.
		
		\emph{Case $1$}:  $0\le y_n < p_{j,n}$. Let $\vec r=\x+\theta_j^{-1}(\y)\in \L_j$ and write $\vec r=\sum_{i=1}^{n}r_i\x_{j,i}+r\c_j$ in normal form. If $0\le \lambda_n<p_{j,n}-y_n$, then we have $\theta_j(\vec r)=\theta_j(\x)+\y$. In this case, $\phi_{j,0}(R(\y))_{\x}=R_{\theta_j(\vec r)+(p_n-p_{j,n})\x_n}$ has a basis $\{\prod_{i=1}^{n-1} X_i^{r_i+d_ip_i} X_n^{r_n+p_n-p_{j,n}}\mid \sum_{i=1}^{n-1} d_i=r,~ d_i\ge 0\}$ and $R^j(\theta_j^{-1}(\y))_{\x}=R_{\vec r}^j$ has a basis $\{\prod_{i=1}^{n-1} X_i^{r_i+d_ip_{i}} X_n^{r_n}\mid \sum_{i=1}^{n-1} d_i=r,~ d_i\ge 0\} $. Thus there is an isomorphism $\zeta_{\x}:\phi_{j,0}(R(\y))_{\x} \to R^j(\theta_j^{-1}(\y))_{\x}$ of $\mathbf{k}$-spaces, sending $ \prod_{i=1}^{n-1} X_i^{r_i+d_ip_i} X_n^{r_n+p_n-p_{j,n}}$  to  $\prod_{i=1}^{n-1} X_i^{r_i+d_ip_i} X_n^{r_n}$. 										
		If  $ p_{j,n}-y_n\le \lambda_n <p_{j,n} $, then we have $\theta_j(\vec r)=\theta_j(\x)+\y+(p_n-p_{j,n})\x_n$. By a similar analysis as above, there exists an isomorphism $\zeta_{\x}:\phi_{j,0}(R(\y))_{\x} \to R^j(\theta_j^{-1}(\y))_{\x}$ of $\mathbf{k}$-spaces, sending $\prod_{i=1}^{n-1} X_i^{r_i+d_ip_i} X_n^{r_n}\in R_{\theta_j(\vec r)}$ to $\prod_{i=1}^{n-1} X_i^{r_i+d_ip_i} X_n^{r_n}\in R^j_{\vec r}$. 			
		
		Therefore we obtain the  isomorphism $\zeta=\bigoplus_{\x\in\L_j}\zeta_{\x}:\phi_{j,0}(R(\y)) \to R^j(\theta_j^{-1}(\y))$ in $\mod^{\L_j} R^j$ as desired.
		
		\emph{Case $2$}: $p_{j,n}\le y_n < p_{n}$. Let $\vec t=\x+\theta_j^{-1}(\y-y_n\x_n)+\c_j$ and write $\vec t=\sum_{i=1}^{n}t_i\x_{j,i}+t\c_j$ in normal form. Then we have $\theta_j(\vec t)=\theta_j(\x)+\y-y_n\x_n+\c$. By a similar analysis as above, there exists an isomorphism 
		$\xi_{\x}:\phi_{j,0}(R(\y))_{\x} \to R^j(\theta_j^{-1}(\y-y_n\x_n)+\c_j)_{\x}$ of $\mathbf{k}$-spaces, sending $\prod_{i=1}^{n-1}X_i^{t_i+d_ip_i}X_n^{t_n+y_n-p_{j,n}}\in R_{\theta_j(\vec t)+(y_n-p_{j,n})\x_n}$ to  $\prod_{i=1}^{n-1}X_i^{t_i+d_ip_i}X_n^{t_n}\in R^j_{\vec t}$, where $t=\sum_{i=1}^{n-1} d_i$ with $d_i \ge 0$. Thus we obtain the  isomorphism $\xi=\bigoplus_{\x\in\L_j}\xi_{\x}:\phi_{j,0}(R(\y)) \to R^j(\theta_j^{-1}(\y-y_n\x_n)+\c_j)$. 
		
		Hence we have the assertion for $k = 0$. For any $k\in \Z$, we set $(\y_n+k)\x_n=a\x_n+b\c$ in normal form. Thus we have $\phi_{j,k}(R(\y))\simeq R^j(\theta_j^{-1}(\y+k\x_n)-k\x_{j,n})=R^j(\theta_j^{-1}(\y-(bp_n-k)\x_n)+(bp_{j,n}-k)\x_{j,n})$ if $0\le a < p_{j,n}$, and $\phi_{j,k}(R(\y))\simeq R^j(\theta_j^{-1}(\y+k\x_n-a\x_n)+\c_j-k\x_{j,n})=R^j	(\theta_j^{-1}(\y-y_n\x_n)+((b+1)p_{j,n}-k)\x_{j,n})$ otherwise. Thus we have the assertion.
	\end{proof}
	
	Recall from Lemma \ref{RI-adjoint triple}(d) that both the  functors $\phi_{j,k}$ and $\psi_{j,k}$ can restrict to the categories of Cohen-Macaulay modules. By Lemma \ref{preserve projective modules}, these restricted exact functors preserve projective modules. Therefore they induce triangle functors on the stable categories of Cohen-Macaulay modules, see \cite[Chapter I, Lemma 2.8]{Hap1}. By abuse of notation, the induced triangle functors are still denoted by $\phi_{j,k}$ and  $\psi_{j,k}$. In this case,  $(\phi_{j,k},\psi_{j,k},\phi_{j,k+1})$ is still an adjoint triple and $\psi_{j,k}$ is still fully faithful, see \cite[Lemma 2.3]{C}.

	\subsection{ $\phi_{j,k}$ and $\psi_{j,k}$ acts on Cohen-Macaulay modules}

	In this subsection, we provide an explicit formula for the action of the reduction/insertion functors to an important  class of Cohen-Macaulay modules. 
	Recall that $\vdelta=\sum_{i=1}^{n}(p_i-2)\x_i$  is called the \emph{dominant element} of $\L$. Note that $0 \leq \x \leq \vdelta$ holds if and only if we have $\x=\sum_{i=1}^n\lambda_i\x_i$ with $0 \le \lambda_i \le p_i-2$. Denote by $\s:=\sum_{i=1}^{n}\x_i$. For each element $\vell=\sum_{i=1}^{n}\ell_i\x_i$ in the interval $[\s,\s+\vdelta]$, let
	\begin{align} \label{important CM module}
		E^{\vell}:=R/(X_i^{\ell_i}\mid 1\le i\le n)\in \mod^{\L} R ~\text{ and } ~U^{\vell}:=\rho(E^{\vell}),
	\end{align}	
	where $\rho$ is the composition $\mod^{\L}R \subset\DDD^{\bo}(\mod^{\L}R)\to\DDD_{\rm sg}^{\L}(R)\xrightarrow{\sim} \underline{\CM}^{\L}R$. 
	The modules $U^{\vell}$ ($\s \le \vell \le \s+\vdelta$), up to degree shift and suspension,  play important roles in this paper, which allow us sufficient grip on the properties of $\underline{\CM}^{\L} R$.
	
	Similarly, we have the notation $\vdelta_j:=\sum_{i=1}^{n-1}(p_{i}-2)\x_{j,i}+p_{j,n}\x_{j,n}$, $\s_j:=\sum_{i=1}^{n}\x_{j,i}$, $E_j^{\vell}:=R^j/(X_i^{\ell_i}\mid 1\le i\le n)$ and $U_j^{\vell}:=\rho(E_j^{\vell})$ for  $\vell=\sum_{i=1}^{n}\ell_i\x_{j,i}\in[\s_j,\s_j+\vdelta_j]$, where $\rho$ is the composition $\mod^{\L_j}R^j \subset \DDD^{\bo}(\mod^{\L_j}R^j)\to\DDD_{\rm sg}^{\L_j}(R^j)\xrightarrow{\sim} \underline{\CM}^{\L_j}R^j$. Note that we have abused notation by calling different functors $\rho$. However, no confusion should arise as they all have different domains and codomains. 
	
	Recall that the functors $\phi_{j,k}$ and $\psi_{j,k}$ from $\mod^{\L} R$ to $\mod^{\L} R$ induce the triangle functors $\phi_{j,k}$ and $\psi_{j,k}$ from $\underline{\CM}^{\L} R$ to $\underline{\CM}^{\L} R$, respectively. Then we have natural isomorphism $\rho\phi_{j,k}\simeq\phi_{j,k}\rho$ and $\rho\psi_{j,k}\simeq\psi_{j,k}\rho$.  
	
	\begin{lemma}\label{L-action to U}  Under the above setting, the following assertions hold.
		\begin{itemize}
			\item[(a)]  Let $\vell=\sum_{i=1}^{n}\ell_i\x_{i}\in[\s,\s+\vdelta]$ and $\y=\sum_{i=1}^{n}y_i\x_{i}+y\c\in \L$. Then $\phi_{j,0}(U^{\vell}(\y))$ is given by one of the following four cases:

			\item[$\bullet$] If $y_n=0$ and $\ell_n>p_n-p_{j,n}$, then  $\phi_{j,0}(U^{\vell}(\y))=U_j^{\theta^{-1}_j(\vell-(p_n-p_{j,n})\x_n)}\theta^{-1}_j(\y);$
			\item[$\bullet$] If $1\le y_n <p_{j,n}$, then  $\phi_{j,0}(U^{\vell}(\y))=U_j^{\theta^{-1}_j(\vell-m\x_n)}\theta^{-1}_j(\y)$, where $m$ denotes the median value of the set $\{0, \ell_n-y_n,p_n-p_{j,n}\};$
			\item[$\bullet$] If $p_{j,n}\le y_n <p_n$ and $y_n-{p_{j,n}}<\ell_n< y_n$, then  $$\phi_{j,0}(U^{\vell}(\y))=U_j^{\theta^{-1}_j(\vell-(y_n-p_{j,n})\x_n)}\theta^{-1}_j(\y-y_n\x_n+\c);$$
			\item[$\bullet$] Otherwise $\phi_{j,0}(U^{\vell}(\y))=0$.\\
			In particular, $\phi_{j,0}(\rho(\mathbf{k})(\y))=\rho(\mathbf{k})(\theta_j^{-1}(\y-\x_n)+\x_{j,n})$ if $1\le y_n\le p_{j,n}$, and zero otherwise.
			\item[(b)]  Let $\vell=\sum_{i=1}^{n}\ell_i\x_{j,i}\in[\s_j,\s_j+\vdelta_j]$ and $\y=\sum_{i=1}^{n}y_i\x_{j,i}+y\c_j\in \L_j$.  Then
			\[\psi_{j,0}(U_j^{\vell}(\y))=\left\{\begin{array}{ll}
				U^{\theta_j(\vell)}(\theta_j(\y))	 &\text{if }~   y_n\ge \ell_n,\\
				U^{\theta_j(\vell)+(p_n-p_{j,n})\x_n}(\theta_j(\y)) &\text{otherwise}.
			\end{array}\right.\]
			In particular, $\psi_{j,0}(\rho(\mathbf{k})(\y))=\rho(\mathbf{k})(\theta_j(\y))$ if $ y_n\ge 1$, and $\psi_{j,0}(\rho(\mathbf{k})(\y))=U^{\s+(p_n-p_{j,n})\x_n}(\theta_j(\y))$ otherwise.
		\end{itemize}
	\end{lemma}
	\begin{proof} We only prove  (a), as (b) can be shown similarly. Without loss of generality, we can assume $\y=y_n\x_n$ with $0\le y_n <p_n$, see the paragraph before Lemma \ref{RI-adjoint triple}. Note that $E^{\vell}=\bigoplus_{0\le \x \le \vell-\s}\mathbf{k}(-\x)$ as an $\L$-graded $R$-module. 
		
		(a)  We consider the following three cases.
		
		\emph{Case $1$}: $y_n=0$. If $\ell_n>p_n-p_{j,n}$, then $\phi_{j,0}(E^{\vell})=E_j^{\theta^{-1}_j(\vell-(p_n-p_{j,n})\x_n)}$ from the construction of $\phi_{j,0}$ and so  $\phi_{j,0}(U^{\vell})=U_j^{\theta^{-1}_j(\vell-(p_n-p_{j,n})\x_n)}$. Otherwise we have $\phi_{j,0}(E^{\vell})=0$ and so  $\phi_{j,0}(U^{\vell})=0$. 
		
		\emph{Case $2$}: $1\le y_n <p_{j,n}$. From the construction of $\phi_{j,0}$, if $ \ell_n \le y_n$, then $\phi_{j,0}(E^{\vell}(\y))=E_j^{\theta^{-1}_j(\vell)}\theta^{-1}_j(\y)$; if $y_n<\ell_n\le y_n+p_n-p_{j,n}$, then $\phi_{j,0}(E^{\vell}(\y))=E_j^{\theta^{-1}_j(\vell-\ell_n\x_n+y_n\x_n)}\theta^{-1}_j(\y)$;
		otherwise  $\phi_{j,0}(E^{\vell}(\y))=E_j^{\theta^{-1}_j(\vell-(p_n-p_{j,n})\x_n)}\theta^{-1}_j(\y)$.
		 By		applying $\rho$ to both sides of all the above equalities and putting things together, the second statement of the assertion (a) follows.
		
		\emph{Case $3$}:  $p_{j,n}\le y_n <p_n$. From the construction of $\phi_{j,0}$, if $\ell_n\le y_n-p_{j,n}$, then $\phi_{j,0}(E^{\vell}(\y))=0$ and so $\phi_{j,0}(U^{\vell}(\y))=0$;
		if $y_n-{p_{j,n}}<\ell_n< y_n$, then  $\phi_{j,0}(E^{\vell}(\y))=E_j^{\theta^{-1}_j(\vell-(y_n-p_{j,n})\x_n)}\theta^{-1}_j(\y-y_n\x_n+\c)$ and thus the third statement of the assertion (a) holds; otherwise since $(X_i^{\ell_i})_{1\le i<n}$ forms an $R$-regular sequence, we have that $\phi_{j,0}(E^{\vell})=R/( X_1^{\ell_1},\dots, X_{n-1}^{\ell_{n-1}}, X_{n}^{p_{n}})=R/( X_1^{\ell_1},\dots, X_{n-1}^{\ell_{n-1}})$ has finite projective dimension.  By applying $\rho$ to both sides, we have $\phi_{j,0}(U^{\vell}(\y))=0$.	
				
		Therefore we obtain the first assertion of (a).	In particular if $\vdelta=0$, we have $\vell=\s$. In this case, $U^{\s}=\rho(\mathbf{k})$ and $\ell_n=1$. If $1\le y_n <p_{j,n}$, then by the second case of (a), 
		$\phi_{j,0}(\rho(\mathbf{k})(\y))=\rho(\mathbf{k})(\theta_j^{-1}(\y))$. If $y_n=p_{j,n}$,  then by the third case of (a), 
		$\phi_{j,0}(\rho(\mathbf{k})(\y))=\rho(\mathbf{k})(\theta^{-1}_j(\y-y_n\x_n+\c))$. Combining these two situations together, we have 
		$\phi_{j,0}(\rho(\mathbf{k})(\y))=\rho(\mathbf{k})(\theta_j^{-1}(\y-\x_n)+\x_{j,n})$ if $1\le y_n\le p_{j,n}$. Otherwise, one can check easily that $\phi_{j,0}(\rho(\mathbf{k})(\y))=0$ by the first assertion of (a).
	\end{proof}
	
	\begin{proposition}\label{U decomp.} 
		For any $\vell=\sum_{i=1}^{n} \ell_i\x_i $ with $\s\le \vell \le \s+\vdelta$,  
		\begin{align*}
			U^{\vell}= 
			\begin{cases}
				{\psi}_{1,p_{1,n}-1} (U_1^{\theta_1^{-1}(\vell)})	 &  \text{if } 1\le\ell_n < p_{1,n},\\
				{\psi}_{2,0} (U_2^{\theta_2^{-1}(\vell-(p_{n}-p_{2,n})\x_{n})}) & \text{if } p_{1,n} \le \ell_n < p_n.
			\end{cases}
		\end{align*} 
		Moreover, we have 
		$$ \bigoplus_{\s \le \vell \le \s+\vdelta} U^{\vell}= \bigoplus_{\s_1 \le \vell \le \s_1+\vdelta_1} {\psi}_{1,p_{1,n}-1}(U_1^{\vell})\oplus \bigoplus_{\s_2 \le \vell \le \s_2+\vdelta_2} {\psi}_{2,0} (U_2^{\vell}).$$
	\end{proposition}
	\begin{proof}
		This is immediate from  Lemma \ref{L-action to U}(b).
	\end{proof}

	\subsection{Proof of Theorem \ref{Thm A}}

	\begin{theorem}	\label{ladder diagram} The following diagram is an infinite ladder of period $p_n$.
	 		\begin{equation}\label{ladder diagram of BP  singularities}
	 		\xymatrix@C=1cm{
	 			\underline{\CM}^{\L_1}{R^1}
	 			\ar@/_3pc/[rrr]_{\vdots}|{{\psi}_{1,q}}
	 			\ar@/^3pc/[rrr]^{\vdots}|{{\psi}_{1,q-2}}
	 			\ar[rrr]|{{\psi}_{1,q-1}}
	 			&&& \underline{\CM}^{\L}{R}
	 			\ar@/_1.5pc/[lll]|{{\phi}_{1,q-1} }
	 			\ar@/^1.5pc/[lll]|{{\phi}_{1,q}}
	 			\ar[rrr]|{{\phi}_{2,0}}
	 			\ar@/_3pc/[rrr]_{\vdots}|{{\phi}_{2,1}}
	 			\ar@/^3pc/[rrr]^{\vdots}|{{\phi}_{2,-1}}
	 			&&& \underline{\CM}^{\L_2}{R^2}
	 			\ar@/_1.5pc/[lll]|{{\psi}_{2,-1}}
	 			\ar@/^1.5pc/[lll]|{{\psi}_{2,0}}
	 		}
	 	\end{equation}
		Here, the functors in the ladder are given by (\ref{Reduction/insertion functor}) and write $q:=p_{1,n}$.
		
	\end{theorem}
	\begin{proof} Recall that for any $j=1,2$ and $k\in \Z$, $(\phi_{j,k},\psi_{j,k},\phi_{j,k+1})$ forms an adjoint triple and $\psi_{j,k}$ is fully faithful.  We claim that ${\phi}_{1,q}{\psi}_{2,0}=0$ and $\thick \langle \Im {\psi}_{1,q-1} \cup \Im {\psi}_{2,0}\rangle =\underline{\CM}^{\L}{R}$. Consequently, we obtain a recollement by \cite[Lemma 2.5]{C}, and further the diagram (\ref{ladder diagram of BP  singularities}) is a ladder.
		
		\emph{Step $1$}: We show that ${\phi}_{1,q}{\psi}_{2,0}= 0$. 		
		By Proposition \ref{thick},  it suffices to show that ${\phi}_{1,q}{\psi}_{2,0}(\rho(\mathbf{k})(\z))=0$ holds for any $\z \in \L_2$. Write $\z=\sum_{i=1}^nz_i\x_{2,i}+z\c_2\in \L_2$ in normal form. Then by Lemma \ref{L-action to U}(b) we have ${\phi}_{1,q}{\psi}_{2,0}(\rho(\mathbf{k})(\z))=0$ if $z_n \ge 1$. If $z_n=0$, then  ${\psi}_{2,0}(\rho(\mathbf{k})(\z))=U^{\s+(p_n-p_{2,n})\x_n}(\theta_2(\z))$ holds
		 by Lemma \ref{L-action to U}(b), and further we have ${\phi}_{1,q}{\psi}_{2,0}(\rho(\mathbf{k})(\z))={\phi}_{1,0}(U^{\s+(q-1)\x_n}(\theta_2(\z)+q\x_n-\c))$.
		 This equals to zero by Lemma \ref{L-action to U}(a). Hence ${\phi}_{1,q}{\psi}_{2,0}= 0$.
		
		\emph{Step $2$}: We show that $\thick \langle \Im {\psi}_{1,q-1} \cup \Im {\psi}_{2,0}\rangle =\underline{\CM}^{\L}{R}$. For simplicity, we write $\mathcal{T}=\thick \langle \Im {\psi}_{1,q-1} \cup \Im {\psi}_{2,0}\rangle$. By Proposition \ref{thick}, it is enough to prove that  $\rho(\mathbf{k})(\y)\in \mathcal{T}$ for any $\y\in \L$. Let $\y=\sum_{i=1}^{n}y_i\x_i+y\c \in \L$ be in normal form. Note that by Lemma \ref{L-action to U}(b), $\rho(\mathbf{k})(\y)={\psi}_{2,0}(\rho(\mathbf{k})(\theta^{-1}_2(\y)))\in \Im {\psi}_{2,0}$ if $1\le y_n<p_{2,n}$. Similarly,  $\rho(\mathbf{k})(\y)$ belongs to $\Im {\psi}_{1,q-1}$ if $p_{2,n}< y_n < p_n$ or $y_n=0$. Hence we have that $\rho(\mathbf{k})(\y)\in \mathcal{T}$ for any $\y\in \L$ with $y_n\neq p_{2,n}$. It remains to deal with  the case  $y_n=p_{2,n}$. 
		Observe that there exists an exact sequence 
		$$ 0 \to K \to E^{\s+(p_{2,n}-1)\x_n}(\y) \to \mathbf{k}(\y) \to 0,$$ 
		in $\mod^{\L} R$, where $K=E^{\s+(p_{2,n}-2)\x_n}(\y-\x_n)$ has composition factors $\{\mathbf{k}(\y-\x_n),\dots,\mathbf{k}(\y-(p_{2,n}-1)\x_n) \}$. This yields a triangle 
		\begin{align}\label{tri}
			 \rho(K) \to U^{\s+(p_{2,n}-1)\x_n}(\y) \to \rho(\mathbf{k})(\y) \to \rho(K)[1]
		\end{align}		
		in $\underline{\CM}^{\L} R$. Then $\rho(K)$ belongs to $\thick \langle \rho(\mathbf{k})(\y-\x_n),\dots,\rho(\mathbf{k})(\y-(p_{2,n}-1)\x_n)  \rangle$. Since $U^{\s+(p_{2,n}-1)\x_n}(\y)={\psi}_{1,q-1}(\rho(\mathbf{k})(\theta_1^{-1}(\y-p_{2,n}\x_n)+\x_{1,n}))\in \mathcal{T}$ by Lemma \ref{L-action to U}, we have  $\rho(\mathbf{k})(\y)\in\mathcal{T}$ by (\ref{tri}). Thus $\thick \langle \Im {\psi}_{1,q-1} \cup \Im {\psi}_{2,0}\rangle =\underline{\CM}^{\L}{R}$.
		
		Therefore we have shown that the diagram (\ref{ladder diagram of BP  singularities}) is a ladder.
		
		\emph{Step $3$}: We show that the ladder is periodic of period $p_n$. Recall that there is a bijection between recollements of $\underline{\CM}^{\L} R$ and TTF triples in $\underline{\CM}^{\L} R$. Observe that for any $0\le m<p_n$ and $b\in \Z$,
		we have 
		\begin{align*}
			\ker \phi_{j,bp_n+m}&=\ker (b(p_{j,n}-p_n)\x_n) \phi_{j,m}=\ker \phi_{j,m},\\
			\Im \psi_{j,bp_n+m}&=\Im  \psi_{j,m}(b(p_n-p_{j,n})\x_n)=\Im \psi_{j,m}.
		\end{align*}
		Hence there is an unbound TTF tuple which is periodic of period $p_n$:
		$$(\cdots, \ker(\phi_{2,-1}), \Im(\psi_{2,-1}), \ker(\phi_{2,0}), \Im(\psi_{2,0}), \ker(\phi_{2,1}),\cdots ).$$
		It follows that the ladder (\ref{ladder diagram of BP  singularities}) is periodic of period $p_n$. This finishes the proof.
	\end{proof}

	\begin{remark}  In the case $n=3$, it is  known that the category ${\CM}^{\L} R$ and the category of vector bundles ${\vect}\, \X$ on a weighted projective line $\X$ equipped  with a distinguished exact structure, are equivalent as Frobenius categories. Moreover, this induces a triangle equivalence between the stable categories of them, that is, $\underline{\CM}^{\L} R\simeq \underline{\vect}\, \X$.  	This means that the ladder (\ref{ladder diagram of BP  singularities}) gives a different interpretation of \cite[Theorem 4.9]{Ruan} in terms of Cohen-Macaulay representations.
	\end{remark}

	\section{Eisenbud periodicity}	\label{sec: Eisenbud periodicity}
		In this section, we give an explicit formula of the action of the suspension functor on  the Cohen-Macaulay modules of form (\ref{important CM module}), which will be used frequently in this paper. This gives immediately a version of Eisenbud periodicity.
		
	\begin{proposition}\label{c=[2]} Let $\vell=\sum_{i=1}^{n}\ell_i\x_i$ with $\s\le \vell \le \s+\vdelta$. Then for any $1\le i \le n$, there exists an isomorphism
		\begin{align} \label{susp. action}
			U^{\vell}[1]=U^{\vell+(p_i-2\ell_i)\x_i}((p_i-\ell_i)\x_i).
		\end{align}	
	In particular, 	$U^{\vell}[2]=U^{\vell}(\c)$ and $U^{\vell}[n]=U^{n\c-\vell}(n\c-\vell)$.
	\end{proposition}
	\begin{proof}Let $\mathcal{S}=\{ (q_1,\dots,q_n)\in \Z^{n} \mid 2\le q_i\le p_i \}$. We  equip $\mathcal{S}$ with the structure of a poset: $(a_1,\dots,a_n)\le (b_1,\dots,b_n)$ if and only if $a_i\le b_i$ for any $1\le i \le n$.	We show the assertion by  induction on $\mathcal{S}$ with respect to the partial order.
		  		
		Assume $p_i=2$ for all $1\le i \le n$. The assertion reduces to the claim that we have to show $\rho(\mathbf{k})[1]=\rho(\mathbf{k})(\x_i)$ for any $1\le i \le n$. 
		Note that there is an exact sequence
		$$0 \to \mathbf{k} \to \mathbf{k}[X_i]/(X_i^2)(\x_i) \to \mathbf{k}(\x_i) \to 0$$
		in $\mod^{\L} R$, and the middle term has finite projective dimension since $(X_k)_{1\le k\le n, \neq i}$ forms an $R$-regular sequence. Thus  $\rho(\mathbf{k})[1]=\rho(\mathbf{k})(\x_i)$.
		
		 Next we consider the induction step. Interchange the roles of two weights that changes neither $R$ nor $\L$,  and thus without loss of generality we can assume $p_n\ge 3$. Suppose that the claim holds for all weight type $(q_1,\dots,q_n)<(p_1,\dots,p_n)$.
		 Let $R^{j}$ be an $\L_j$-graded BP singularity  with weights $p_1,\ldots,p_{n-1},p_{j,n}$ for $j=1,2$, where $p_{1,n}=p_n-1$ and $p_{2,n}=2$.  
		Recall from Proposition \ref{U decomp.} that we have
		\begin{align*}
			U^{\vell}= 
			\begin{cases}
				{\psi}_{1,p_n-2} (U_1^{\theta_1^{-1}(\vell)})	 &  \text{if } 1\le\ell_n \le p_n-2,\\
				{\psi}_{2,0} (U_2^{\theta_2^{-1}(\vell-(p_n-2)\x_{n})}) & \text{if } \ell_n = p_n-1.
			\end{cases}
		\end{align*} 
		By the induction hypothesis, the assertion is trivial for $0\le i <n$. 	
		Next we  show the assertion for $i=n$ and have to deal with the following two cases. 
		
		\emph{Case $1$}: $1\le\ell_n \le p_n-2$.  By the induction hypothesis, $$U_1^{\theta_1^{-1}(\vell)}[1]=U_1^{\theta_1^{-1}(\vell)+(p_n-1-2\ell_n)\x_{1,n}}((p_n-1-\ell_n)\x_{1,n}).$$
		Applying ${\psi}_{1,p_n-2}$ to both sides, by Lemma \ref{L-action to U}(b), we have
		\begin{align*}
		U^{\vell}[1]	&={\psi}_{1,p_n-2} (U_1^{\theta_1^{-1}(\vell)+(p_n-1-2\ell_n)\x_{1,n}}((p_n-1-\ell_n)\x_{1,n}))\\
			&=(-(p_n-2)\x_n){\psi}_{1,0} (U_1^{\theta_1^{-1}(\vell)+(p_n-1-2\ell_n)\x_{1,n}}((p_n-2-\ell_n)\x_{1,n}+\c_1))\\
			&=U^{\vell+(p_n-2\ell_n)\x_n}((p_n-\ell_n)\x_n).
		\end{align*}
		
		\emph{Case $2$}: $\ell_n = p_n-1$. By the induction hypothesis,
		$$U_2^{\theta_2^{-1}(\vell-(p_n-2)\x_n)}[1]=U_2^{\theta_2^{-1}(\vell-(p_n-2)\x_n)}(\x_{2,n}).$$
		Applying ${\psi}_{2,0}$ to both sides, we have $U^{\vell}[1]=U^{\vell-(p_n-2)\x_n}(\x_n)$ by Lemma \ref{L-action to U}(b), which is equal to $U^{\vell+(p_n-2\ell_n)\x_n}((p_n-\ell_n)\x_n)$. This finishes the induction step and thus the first assertion follows. 
		
		Applying (\ref{susp. action}) repeatedly twice, we have	
		\begin{align*}
			U^{\vell}[2]	&=U^{\y}[1]((p_i-\ell_i)\x_i)\\
			&=U^{\y+(p_i-2y_i)\x_i}((p_i-\ell_i+p_i-y_i)\x_i)\\
			&=U^{\vell}(\c).
		\end{align*}
		Here $\y:=\vell+(p_i-2\ell_i)\x_i$ with $1\le i\le n$, and thus $y_i:=l_i(\y)=p_i-\ell_i$. 
		
		Applying (\ref{susp. action}) in the order of the index set $\{1,\dots,n\}$, we have
		\begin{align*}
			U^{\vell}[n]	&=U^{\vell+(p_1-2\ell_1)\x_1}[n-1]((p_1-\ell_1)\x_1)\\
			&=U^{\vell+(p_1-2\ell_1)\x_1+(p_2-2\ell_2)\x_2}[n-2]((p_1-\ell_1)\x_1+(p_2-\ell_2)\x_2)\\
			&=\dots \\
			&=U^{n\c-\vell}(n\c-\vell).
		\end{align*}
		Therefore we have the assertion.
	\end{proof}
	
		This gives immediately a version of Eisenbud periodicity, see \cite{Y}.
	
	\begin{corollary}[Eisenbud periodicity] There exists an isomorphism $[2]\simeq (\c)$ of functors $\underline{\CM}^{\L} R \to \underline{\CM}^{\L} R$.
	\end{corollary}
	\begin{proof} By Proposition \ref{c=[2]}, we have $\rho(\mathbf{k})[2]=\rho(\mathbf{k})(\c)$. 
		Then	 the assertion follows immediately from Proposition \ref{thick}.
	\end{proof}
	
	As an immediate consequence of Proposition \ref{c=[2]},  we have the following two special cases, which are frequently used later. 
	
	\begin{corollary} \label{x_i=[1]} Let $\vell=\sum_{i=1}^{n} \ell_i\x_i$ with $\s\le \vell \le \s+\vdelta$.
		\begin{itemize}
			\item[(a)]  If one of weight $p_i=2$, then $U^{\vell}[1]=U^{\vell}(\x_i)$.	
			\item[(b)]	 For any $1\le i \le n$,  $U^{\vell+(p_i-\ell_i-1)\x_i}[1]=U^{\vell-(\ell_i-1)\x_i}(\x_i)$. 
			In particular,	 $U^{\s+(p_i-2)\x_i}[1]=\rho(\mathbf{k})(\x_i)$ and $U^{\s+\vdelta}=\rho(\mathbf{k})(\s)[-n]$.
		\end{itemize}
	\end{corollary}

	\section{Tilting objects via tensor products}\label{sec: Tilting objects via tensor products}
	The aim of this section is to study tilting theory in the triangulated category $\underline{\CM}^{\L} R$. We construct a class of tilting objects in $\underline{\CM}^{\L} R$ via our ladder (\ref{ladder diagram}), called extended tilting $n$-cuboids, including two important tilting objects that have been widely studied.
	We show that endomorphism algebra of such a tilting object is $n$-fold tensor products of homogeneous acyclic Nakayama algebras.

	\subsection{Tensor products of homogeneous acyclic Nakayama algebras}
	Recall that a finite dimensional algebra is called \emph{Nakayama algebra} if its indecomposable projective or injective modules are uniserial, that is, it has a unique composition series. A natural class of such algebras are \emph{homogeneous acyclic Nakayama algebras}, which have the forms $\mathbf{k}\vec{\AA}_{n}(m):=\mathbf{k}\vec{\AA}_n/{\rm rad^m} \, \mathbf{k}\vec{\AA}_n$ for  $n,m \ge 1$. Here $\mathbf{k}\vec{\AA}_n$ denotes the path algebra of the equioriented quiver of type $\AA_n$ and ${\rm rad}\, \mathbf{k}\vec{\AA}_n$ denotes the ideal generated by all arrows of $\mathbf{k}\vec{\AA}_n$.
	
		 We introduce a class of finite dimensional algebras by modifying the definition of CM-canonical algebras in \cite{HIMO}, which gives a new quiver description of the tensor product of homogeneous acyclic Nakayama algebras. Here $\vdelta =\sum_{i=1}^{n}(p_i-2)\x_i$.
	
	\begin{definition}  Let $\mathbf{q}=(q_1,\dots,q_n)$ be an $n$-tuple of integers with $1\le q_i\le p_i-1$ and
		$S:=R/(X_i^{q_i} \mid 1\le i \le n)$. Denote by $I=[0,\vdelta]$. We define a matrix algebra
		$$\Lambda(\mathbf{q}):=(S_{\x-\y}){}_{\x,\y\in I},$$
		where  the multiplication of $\Lambda(\mathbf{q})$ is given by
		\[(a_{\x,\y})_{\x,\y\in I}\cdot(a'_{\x,\y})_{\x,\y\in I}:=
		(\sum_{\z\in I}a_{\x,\z}\cdot a'_{\z,\y})_{\x,\y\in I}.\]		
	\end{definition}

	\begin{lemma} \label{iso. algs} In the setup above,  there is an isomorphism of k-algebras
		$$\Lambda(\mathbf{q}) \simeq \bigotimes_{1\le i \le n} \mathbf{k}\vec{\AA}_{p_i-1}(q_i).$$ 
	\end{lemma}
	
	\begin{proof}  Let $0\le \x,\y \le \vdelta$ with $\x=\sum_{i=1}^{n}\lambda_i\x_i$ and $\y=\sum_{i=1}^{n}y_i\x_i$ in the normal forms.
		We equip $[0,\vdelta]$ with the structure of a linear order set: 
		$$\x\preceq \y  \ \ \text{if and only if} \ \ (\lambda_1,\dots,\lambda_n)\le_{\rm lex} (y_1,\dots,y_n),$$ where $\le_{\rm lex}$ denotes the usual  \emph{lexicographic order}, that is, either $\lambda_i=y_i$ for all $1\le i \le n$, or that there exists $1\le k \le n$ such that $\lambda_j=y_j$ for all $j< k$ and  $\lambda_k<y_k$.  
		Arranging $[0,\vdelta]$ in this order, it is easy to check that
		\begin{align*}
			\Lambda(\mathbf{q}) &\simeq \bigotimes_{1\le i \le n}(\mathbf{k}I_{p_i-1}+X_iJ_{p_i-1}+\dots+X_i^{q_i-1}J_{p_i-1}^{q_i-1})\\
			& \simeq \bigotimes_{1\le i \le n} \mathbf{k}\vec{\AA}_{p_i-1}(q_i),
		\end{align*}
		where $I_m$  denotes the identity  matrix of size $m$ and $J_m:=\left(\begin{smallmatrix} 0 & I_{m-1} \\ 0 & 0  \end{smallmatrix}\right)$.		
	\end{proof}	 
	
	Note that for the case $\mathbf{q}=(p_1-1,\dots,p_n-1)$, the algebra $\Lambda(\mathbf{q})\simeq \bigotimes_{1\le i \le n} \mathbf{k}\vec{\AA}_{p_i-1}$,
	which reduces to the CM-canonical algebra of the $\L$-graded BP singularity $R$.
	
	\begin{proposition} \label{shape theorem}
		The $\mathbf{k}$-algebra $\Lambda(\mathbf{q})$ is presented
		by the quiver $Q$ defined by vertices $Q_0:=[0, \vdelta]$ and arrows
		$$Q_1:=\{\x\xrightarrow{x_i} \x+\x_i \mid 1\le i \le n \ \text{and} \ \x, \x+\x_i\in [0, \vdelta]  \}$$
		with the following relations:
		\begin{itemize}
			\item  $x_ix_j=x_jx_i:\x \to \x+\x_i+\x_j$, where $1\le i,j \le n$ and $0\le \x \le \x+\x_i+\x_j \le \vdelta$,
			\item  $x_i^{q_i}:\x \to \x+q_i\x_i$, where $1\le i \le n$ and $0\le \x \le \x+q_i\x_i \le \vdelta$.
		\end{itemize}
	\end{proposition}

	\begin{proof} The vertex $\x$ of $Q$  corresponds to the primitive idempotent $e_{\x}$ of $\Lambda:=\Lambda(\mathbf{q})$ and the arrow $x_i$ of $Q$  corresponds to the generator $X_i$ of $S$ for $1\le i \le n$. Thus there exists a morphism $\mathbf{k}Q \to \Lambda$ of $\mathbf{k}$-algebras extending by these correspondences, which is surjective. The commutativity relations $X_iX_j=X_jX_i$ are satisfied in $S$. Also the relations $X_i^{q_i}=0$ in $S$  correspond to the relations $x_i^{q_i}=0$ in $\mathbf{k}Q$. Thus there is a surjective morphism $\mathbf{k}Q/I\to \Lambda$, where $I=(x_ix_j-x_jx_i, x_i^{q_i}\mid 1\le i,j \le n).$ Indeed, this is a $\mathbf{k}$-algebra isomorphism  since it clearly  induces an isomorphism	$$e_{\x}(\mathbf{k}Q/I)e_{\y}\simeq e_{\x}\Lambda e_{\y}= S_{\x-\y}$$ for any $\x,\y \in Q_0$, as $\mathbf{k}$-vector spaces.	Hence we have the assertion.
	\end{proof}

	\begin{example} \label{example} Let $n=2$ and assume that $R$ be an $\L$-graded BP singularity of type $(3,4)$. Let $\mathbf{q}=(2,2)$ and thus  $S=\mathbf{k}[X_1,X_2]/(X_1^2,X_2^2)$. By Proposition \ref{shape theorem}, $\Lambda(\mathbf{q})$ is given by the following quiver  with relations. Here,  the dotted lines indicate the relations.
		
		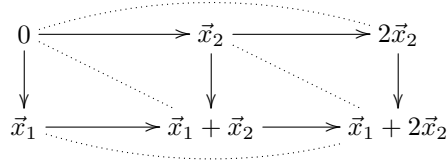
\begin{figure}[H]
			\centering 
			$
			\begin{xy} 0;<7pt,0pt>:<0pt,3.5pt>:: 
				(0,0) *+{\x_1} ="0",(5,0)  ="3", (0,10) *+{0} ="2",  (15,0) ="4", (5,10) ="5",  (15,10) ="6",
				(10,0) *+{\x_1+\x_2} ="1",(10,10) *+{\x_2} ="12",
				(20,0) *+{\x_1+2\x_2} ="11",(20,10) *+{2\x_2} ="112",
				"2", {\ar"0"},"12", {\ar"1"},"12", {\ar"112"},"112", {\ar"11"},
				"0", {\ar"1"},"2", {\ar"12"},"1", {\ar"11"},
				"0", {{\ar@[black]@/_1pc/@{.}"11"}},	
				"2", {{\ar@{.}"1"}},	
				"12", {{\ar@{.}"11"}},	
				"2", {{\ar@[black]@/^1pc/@{.}"112"}},		
			\end{xy}
			$
			\caption{The quiver of $\Lambda(\mathbf{q})$ for $R(3,4)$ and $\mathbf{q}=(2,2)$}
			\label{figure (2,2)}
		\end{figure}
		
	\end{example}
	
	\subsection{The extended tilting $n$-cuboid} 	
	Throughout this subsection, let $R$ be an $\L$-graded BP singularity with weights $p_1,\ldots, p_n$. 
	Recall that  $\vdelta =\sum_{i=1}^{n}(p_i-2)\x_i$ and $\s=\sum_{i=1}^{n}\x_i$. For each $\vell=\sum_{i=1}^{n}\ell_i\x_i\in[\s,\s+\vdelta]$, let
	$$U^{\vell}=\rho(R/(X_i^{\ell_i}\mid 1\le i\le n)),$$ 
	where $\rho$ denotes the composition $\mod^{\L}R \subset \DDD^{\bo}(\mod^{\L}R)\to\DDD_{\rm sg}^{\L}(R)\xrightarrow{\sim} \underline{\CM}^{\L}R$.			
	For each $\x=\sum_{i=1}^{n} \lambda_i\x_i$ with $0\le \x\le \vdelta$, let $ \sigma(\x):= \sum_{i=1}^{n} \lambda_i$. For a subset  $I$ of $\{1,\dots,n\}$, we set $\vdelta_I:=\sum_{i\in I}(p_i-2)\x_i$ and $I^{c}:=\{1,\dots,n\}\setminus I$. 
	Denote by $[1]$  the suspension functor in the triangulated category $\underline{\CM}^{\L} R$.

	\begin{lemma}\label{Hom} For any $\s\le \vell,\z \le \s+\vdelta_I$ and $0\le \x, \y\le \vdelta-\vdelta_{I}$, we have 
		\begin{eqnarray*}
			\underline{\Hom}(U^{\vell}(\x)[-\sigma(\x)], U^{\z}(\y)[-\sigma(\y)])=
			\begin{cases}
				\mathbf{k} &  \text{if } \vell\ge \z \text{ and } 0\le \x-\y \le \s,\\
				0& \text{otherwise.}
			\end{cases}
		\end{eqnarray*}
	\end{lemma}
	\begin{proof}	We endow $\mathcal{S}:=\{ (q_1,\dots,q_n)\in \Z^{n} \mid 2\le q_i\le p_i \}$ with a partial order defined componentwise, that is,  $(a_1,\dots,a_n)\le (b_1,\dots,b_n)$ if and only if $a_i\le b_i$ for any $1\le i \le n$.
		We use induction on $\mathcal{S}$ with respect to the partial order.  
				
		Assume $p_i=2$ for all $1\le i \le n$. The assertion reduces to the claim that we have to show $\underline{\End}(\rho(\mathbf{k}))=\mathbf{k}$. By Kn\"orrer periodicity, there is a triangle equivalence $\CM^{\L} R\simeq \CM^{\L'} R'$, where $R':=\mathbf{k}[Y]/(Y^2)$ is an $\L'$-graded BP singularity with weight two. Thus $\underline{\End}(\rho(\mathbf{k})) \simeq {\End}_{\underline{\CM}^{\L'} R'}(\rho(\mathbf{k}))=\mathbf{k}$.    
			
		Next we consider the induction step. Interchange the roles of two weights that changes neither $R$ nor $\L$,  and thus without loss of generality we can assume $p_n\ge 3$. Assume that $R^{j}$ be an $\L_j$-graded BP singularity  with weights $p_1,\ldots,p_{n-1},p_{j,n}$ for $j=1,2$, where $p_{1,n}=p_n-1$ and $p_{2,n}=2$. We deal with the expression $H:=\underline{\Hom}(U^{\vell}(\x)[-\sigma(\x)], U^{\z}(\y)[-\sigma(\y)])$ depending on whether $n\in I$, and thus we need discuss the following two cases.
		
		\emph{Case $1$}:  $n\in I$.  By Theorem \ref{ladder diagram}, we obtain a piece of ladder
		
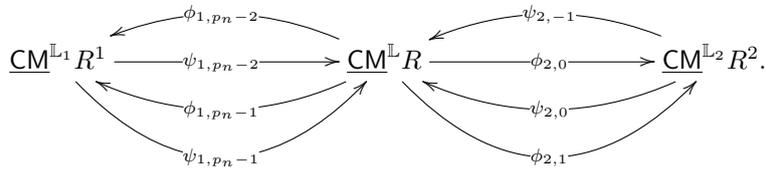
\begin{align*}
			\xymatrix{
				\underline{\CM}^{\L_1}{R^1}
				\ar[rrr]|{{\psi}_{1,p_n-2}}
				\ar@/_3pc/[rrr]|{{\psi}_{1,p_n-1}}
				&&& \underline{\CM}^{\L}{R}
				\ar@/_1.5pc/[lll]|{{\phi}_{1,p_n-2} }
				\ar@/^1.5pc/[lll]|{{\phi}_{1,p_n-1}}
				\ar[rrr]|{{\phi}_{2,0}}
				\ar@/_3pc/[rrr]|{{\phi}_{2,1}}
				&&& \underline{\CM}^{\L_2}{R^2}.
				\ar@/_1.5pc/[lll]|{{\psi}_{2,-1}}
				\ar@/^1.5pc/[lll]|{{\psi}_{2,0}}
			}
		\end{align*} Write $\ell_n:=l_n(\vell)$ and $z_n:=l_n(\z)$. Immediate from Proposition \ref{U decomp.}, we have		
		\begin{align*}
		U^{\vell}(\x)[-\sigma(\x)]= 
			\begin{cases}
			 {\psi}_{1,p_n-2} (U_1^{\theta_1^{-1}(\vell)}({\theta_1^{-1}(\x)})[-\sigma(\x)] )	 &  \text{if } 1\le\ell_n \le p_n-2,\\
				{\psi}_{2,0} (U_2^{\theta_2^{-1}(\vell-(p_n-2)\x_n)}({\theta_2^{-2}(\x)})[-\sigma(\x)] ) & \text{if } \ell_n = p_n-1.
			\end{cases}
		\end{align*} 
		 By the induction hypothesis, the assertion holds if $1\le \ell_n,z_n \le p_n-2$ or $\ell_n=z_n=p_n-1$.	 Next we need deal with the following two cases: (a) $1\le \ell_n \le p_n-2$ and $z_n=p_n-1$; (b) $ \ell_n=p_n-1 $ and $1\le z_n \le p_n-2$. Concerning (a), we have $H=0$ since ${\phi}_{1,p_n-1}{\psi}_{2,0}=0$. Concerning (b), by the second case of Lemma \ref{L-action to U}(a),  
		\begin{align*}
			H&={\Hom}_{\underline{\CM}^{\L_2} R^2}(U_2^{\theta_2^{-1}(\vell-(p_n-2)\x_n)}(\theta_2^{-1}(\x))[-\sigma(\x)],{\phi}_{2,1}(U^{\z}(\y)[-\sigma(\y)]))\\
			&={\Hom}_{\underline{\CM}^{\L_2} R^2}(U_2^{\theta_2^{-1}(\vell-(p_n-2)\x_n)}(\theta_2^{-1}(\x))[-\sigma(\x)],
			U_2^{\theta_2^{-1}(\z-z_n\x_n+\x_n)}(\theta_2^{-1}(\y))[-\sigma(\y)]).
		\end{align*}
		Since $\ell_n=p_n-1$, $\vell \ge \z$ is equivalent to $\vell-(p_n-2)\x_n \ge \z-z_n\x_n+\x_n$. By the induction hypothesis, the assertion holds for the case (b).
		 
		\emph{Case $2$}:  $n\notin I$. By Theorem \ref{ladder diagram}, we obtain a recollement 
		\begin{equation*}
			\xymatrix{
				\underline{\CM}^{\L_1}{R^1}
				\ar[rrr]|{{\psi}_{1,1}}
				&&& \underline{\CM}^{\L}{R}
				\ar@/_1.5pc/[lll]|{{\phi}_{1,1} }
				\ar@/^1.5pc/[lll]|{{\phi}_{1,2}}
				\ar[rrr]|{{\phi}_{2,3-p_n}}
				&&& \underline{\CM}^{\L_2}{R^2}.
				\ar@/_1.5pc/[lll]|{{\psi}_{2,2-p_n}}
				\ar@/^1.5pc/[lll]|{{\psi}_{2,3-p_n}}
			}
		\end{equation*}%
		Write $\x=\sum_{i=1}^{n}\lambda_i\x_i$ and $\y=\sum_{i=1}^{n}y_i\x_i$ in normal form. We claim that  
		\begin{align*}
			U^{\vell}(\x)[-\sigma(\x)]= 
			\begin{cases}
				{\psi}_{1,1}(	U_1^{\theta_1^{-1}(\vell)}(\theta_1^{-1}(\x))[-\sigma(\x)])	 &  \text{if } 0\le \lambda_n \le p_n-3,\\
				{\psi}_{2,3-p_n}(U_2^{\theta_2^{-1}(\vell)}(\theta_2^{-1}(\x-\lambda_n\x_n))[-\sigma(\x-\lambda_n\x_n)])  & \text{if } \lambda_n=p_n-2.
			\end{cases}
		\end{align*} 
		For the case $0\le \lambda_n \le p_n-3$, it follows directly from Lemma \ref{L-action to U}(b). 
		Next we consider the remaining case $\lambda_n=p_n-2$. Since $p_{2,n}=2$,  we have $U_2^{\theta_2^{-1}({\vell})}(\x_{2,n})=U_2^{\theta_2^{-1}({\vell})}[1]$ by Corollary \ref{x_i=[1]}.
		Then $U_2^{\theta_2^{-1}({\vell})}((3-p_n)\x_{2,n})=U_2^{\theta_2^{-1}({\vell})}(\x_{2,n})[2-p_n]$. 
		It follows from Lemma \ref{L-action to U}(b) that we have 		
		\begin{align*}
			&\quad{\psi}_{2,3-p_n}(U_2^{\theta_2^{-1}(\vell)}(\theta_2^{-1}(\x-(p_n-2)\x_n))[-\sigma(\x-(p_n-2)\x_n)])\\
			&=((p_n-3)\x_{n}){\psi}_{2,0}(U_2^{\theta_2^{-1}(\vell)}(\theta_2^{-1}(\x-(p_n-2)\x_n)+\x_{2,n})[-\sigma(\x)])\\
			&=U^{\vell}(\x)[-\sigma(\x)].
		\end{align*}
		By the induction hypothesis, the assertion holds if $0\le \lambda_n,y_n \le p_n-3$ or $\lambda_n=y_n=p_n-2$.	 Next we need deal with the following two cases: ($\alpha$) $0\le \lambda_n \le p_n-3$ and $y_n=p_n-2$; ($\beta$) $\lambda_n=p_n-2$ and $0\le y_n \le p_n-3$. 	Concerning ($\alpha$), we have $H=0$ since ${\phi}_{1,2}{\psi}_{2,3-p_n}=0$. Concerning ($\beta$), by the third case of Lemma \ref{L-action to U}(a), 
		\begin{align*}
			H&={\Hom}_{\underline{\CM}^{\L_1} R^1}({{\phi}_{1,1}}(U^{\vell}(\x)[-\sigma(\x)]),	U_1^{\theta_1^{-1}(\z)}(\theta_1^{-1}(\y))[-\sigma(\y)])\\
			&={\Hom}_{\underline{\CM}^{\L_1} R^1}(U_1^{\theta_1^{-1}(\vell)}(\theta_1^{-1}(\x-\lambda_n\x_n+\c)-\x_{1,n})[-\sigma(\x)],	U_1^{\theta_1^{-1}(\z)}(\theta_1^{-1}(\y))[-\sigma(\y)]).
		\end{align*}
		Notice that $\theta_1^{-1}(\x-\lambda_n\x_n+\c)-\x_{1,n}=\sum_{i=1}^{n}\lambda_i\x_{1,n}$ and $\theta_1^{-1}(\y)=\sum_{i=1}^{n}y_i\x_{1,n}$. Then $0\le \x-\y \le \s$ implies that $0\le \theta_1^{-1}(\x-\lambda_n\x_n+\c)-\x_{1,n}-\theta_1^{-1}(\y) \le \s_1$. By the induction hypothesis, the assertion holds for the case ($\beta$).		
	\end{proof}
		
		Recall that for a triangulated category $\mathcal{T}$, an object $T \in \mathcal{T}$ is \emph{tilting} if
		\begin{itemize}
			\item[(a)] $T$ is \emph{rigid}, that is $\Hom_{\mathcal{T}}(T,T[m])=0$ for all $m \neq 0$.
			\item[(b)] $T$ generates $\mathcal{T}$, that is, $\thick T=\mathcal{T}$, where $\thick T$ denotes by the smallest thick subcategory of $\mathcal{T}$ containing $T$. 					
		\end{itemize}	
			
		  An object $E$ is called \emph{exceptional} in $\mathcal{T}$ if $\End_{\mathcal{T}}(E) = \mathbf{k}$ and $E$ is rigid. Moreover, a sequence of exceptional objects $(E_1,\dots,E_r)$ in $\mathcal{T}$ is called an  \emph{exceptional sequence} if $\Ext_{\mathcal{T}}(E_i,E_j[m])=0$ holds for all  $i>j$  and $m\in \Z$. It is called \emph{strong} if  $\Ext_{\mathcal{T}}(E_i,E_j[m])=0$ holds for all  $i\neq j$ and $m\in \Z$. An exceptional sequence is called \emph{full} if it generates $\mathcal{T}$, that is, $\thick (\bigoplus_{i=1}^{r}E_i)=\mathcal{T}$. 
	
		Gluing techniques for tilting objects in triangulated categories  have been widely studied, see for example \cite{AKL,DZ,LVY,SZ}. 
		The following observation provides a useful and simple method to glue tilting objects for BP singularities.
	
			\begin{proposition} \label{gluing tilting} \cite[Theorem 5.2]{DZ} Let $(\mathcal{T}_1, \mathcal{T}, \mathcal{T}_2)$ be a recollement of form (\ref{recollement}). Assume that $T_1$ and $T_2$ are tilting objects in $\mathcal{T}_1$ and $\mathcal{T}_2$, respectively. Then the following conditions are equivalent.
			\begin{itemize}
				\item[(a)] $T=i_{\ast}(T_1)\oplus j_*(T_2)$ is a tilting object in $\mathcal{T}$.
				\item[(b)] $\Hom_{\mathcal{T}_1}(i^*j_*(T_2),T_1[m])=0$ for any integer $m\neq 0$.
			\end{itemize}
			If moreover $j_{\ast}$ admits a right adjoint $j^{\sharp}$, then the
			following condition is also equivalent to the ones above.
			\begin{itemize} \item[($b'$)] $\Hom_{\mathcal{T}_2}(T_2,j^{\sharp}i_{\ast}(T_1)[m])=0$ for any integer $m\neq 0$.
			\end{itemize}		
		\end{proposition}

			 	Our  main result in this section is the following.
		
		\begin{theorem} \label{tilting object} The object
			$$V=V_I:=\bigoplus_{\s\leq\vell\leq\s+\vdelta_I}\bigoplus_{0\leq\x\leq\vdelta-\vdelta_{I}}U^{\vell}(\x)[-\sigma(\x)]$$	
			is   tilting  in $\underline{\CM}^{\L} R$, called the \emph{extended tilting $n$-cuboid} with endomorphism algebra 
			\begin{align}\label{shape of V}
			\underline{\End}(V)^{\rm op}\simeq \, \bigotimes_{ i\in  I} \mathbf{k}\vec{\AA}_{p_i-1} \otimes  \bigotimes_{i\in I^{c}} \mathbf{k}\vec{\AA}_{p_i-1}(2).
			\end{align}
			Moreover, the objects $U^{\vell}(\x)[-\sigma(\x)]$ $(\s\leq\vell\leq\s+\vdelta_I,\, 0\le \x \le \vdelta-\vdelta_I)$ can be arranged to form a full, strong exceptional sequence.	
		\end{theorem}

		\begin{proof}
			We endow $\mathcal{S}:=\{ (q_1,\dots,q_n)\in \Z^{n} \mid 2\le q_i\le p_i \}$ with a partial order defined componentwise. 
			
		\emph{Step $1$}: We show that $V$ is tilting in $\underline{\CM}^{\L} R$ by induction on $\mathcal{S}$.  
		
		Assume $p_i=2$ for all $1\le i \le n$. Clearly $V=\rho(k)$ is  a tilting object  in $\underline{\CM}^{\L} R$. In fact, $V=\rho(\mathbf{k})$ is a rigid object in $\underline{\CM}^{\L} R$ by Kn\"orrer periodicity, and $\thick \langle \rho(k) \rangle=\underline{\CM}^{\L} R$ by Lemma \ref{thick} and Corollary \ref{x_i=[1]}. 	
					
		 Next we consider the induction step. Interchange the roles of two weights that changes neither $R$ nor $\L$,  and thus without loss of generality we can assume $p_n\ge 3$. Suppose that the claim holds for all weight type $(q_1,\dots,q_n)<(p_1,\dots,p_n)$. Let $R^{j}$ be an $\L_j$-graded BP singularity  with weights $p_1,\ldots,p_{n-1},p_{j,n}$ for $j=1,2$, where $p_{1,n}=p_n-1$ and $p_{2,n}=2$. 
	 By the induction hypothesis,  for $j=1,2$, the object 
		 $$T_j= \bigoplus_{\s_j\le \vell \le \s_j+\vdelta_{j_I}} \bigoplus_{0\le \x \le \vdelta_{j}-\vdelta_{j_I}} U_j^{\vell}(\x)[-\sigma(\x)]$$ 	
		 is tilting in $\underline{\CM}^{\L_j}{R^j}$,  where  
		 $\vdelta_{j_I}:=\sum_{i\in I}(p_{j,i}-2)\x_{j,i}$  with $p_{j,i}:=p_i$ for $1\le i <n$.
		Depending on whether $n\in I$, we have to consider the following two cases.

		\emph{Case $1$}:  $n\in I$.  In this case,  $\vdelta_{1_I}=\sum_{i\in I\setminus\{n\}}(p_{i}-2)\x_{1,i}+(p_{1,n}-2)\x_{1,n}$ and $\vdelta_{2_I}=\sum_{i\in I\setminus\{n\}}(p_{i}-2)\x_{2,i}$.  By Theorem \ref{ladder diagram}, this yields a recollement
		\begin{equation*}
			\xymatrix{
				\underline{\CM}^{\L_1}{R^1}
				\ar[rrr]|{{\psi}_{1,p_n-2}}
				&&& \underline{\CM}^{\L}{R}
				\ar@/_1.5pc/[lll]|{{\phi}_{1,p_n-2} }
				\ar@/^1.5pc/[lll]|{{\phi}_{1,p_n-1}}
				\ar[rrr]|{{\phi}_{2,0}}
				&&& \underline{\CM}^{\L_2}{R^2}.
				\ar@/_1.5pc/[lll]|{{\psi}_{2,-1}}
				\ar@/^1.5pc/[lll]|{{\psi}_{2,0}}
			}
		\end{equation*} 	Note that $\theta_1(\vdelta_{1_I})=\vdelta_{I}-\x_n$ and $\theta_2(\vdelta_{2_I})=\vdelta_{I}-(p_n-2)\x_n$.	  It further follows from Lemma \ref{L-action to U}(b) immediately that we have the equalities 
		\begin{align*}
			V_1&:={\psi}_{1,p_n-2}(T_1)=\bigoplus_{\s\le \vell \le \s+(\vdelta_{I}-\x_n)} \bigoplus_{0\le \x \le \vdelta-\vdelta_{I}} U^{\vell}(\x)[-\sigma(\x)],\\
			V_2&:={\psi}_{2,0}(T_2)=\bigoplus_{\s+(p_n-2)\x_n\le \vell \le \s+\vdelta_{I}} \bigoplus_{0\le \x \le \vdelta-\vdelta_{I}} U^{\vell}(\x)[-\sigma(\x)].
		\end{align*}
		Observe that $[\s,\s+\vdelta_{I}-\x_n]\cup [\s+(p_n-2)\x_n, \s+\vdelta_{I}]=[\s,\s+\vdelta_I]$. Thus $V=V_1\oplus V_2$.
		Following the second case of Lemma \ref{L-action to U}(a), we have 
		$$T_2':={\phi}_{1,p_n-2}{\psi}_{2,0}(T_2)=\bigoplus_{\s_1+(p_n-3)\x_{1,n}\le \vell \le \s_1+\vdelta_{1_I}} \bigoplus_{0\le \x \le \vdelta_1-\vdelta_{1_I}} U_1^{\vell}(\x)[-\sigma(\x)],$$
		which belongs to $\add T_1$.  
		The rigidity of $T_1$ implies that $\Hom_{\underline{\CM}^{\L_1} R^1}(T_2',T_1[m])=0$ for any $m\neq 0$, and thus $V$ is tilting  in $\underline{\CM}^{\L}{R}$ by Proposition \ref{gluing tilting}.
		
		\emph{Case $2$}: $n\notin I$.  In this case, $\vdelta_{1_I}=\sum_{i\in I}(p_{i}-2)\x_{1,i}$  and $\vdelta_{2_I}=\sum_{i\in I}(p_{i}-2)\x_{2,i}.$ 
		By Theorem \ref{ladder diagram}, we obtain a piece of ladder 
		\begin{equation*}
			\xymatrix{
				\underline{\CM}^{\L_1}{R^1}
				\ar[rrr]|{{\psi}_{1,1}}
				\ar@/_3pc/[rrr]|{{\psi}_{1,2}}
				&&& \underline{\CM}^{\L}{R}
				\ar@/_1.5pc/[lll]|{{\phi}_{1,1} }
				\ar@/^1.5pc/[lll]|{{\phi}_{1,2}}
				\ar[rrr]|{{\phi}_{2,3-p_n}}
				\ar@/_3pc/[rrr]|{{\phi}_{2,4-p_n}}
				&&& \underline{\CM}^{\L_2}{R^2}.
				\ar@/_1.5pc/[lll]|{{\psi}_{2,2-p_n}}
				\ar@/^1.5pc/[lll]|{{\psi}_{2,3-p_n}}
			}
		\end{equation*}%
		We claim that
		\begin{align}
		W_1&:={\psi}_{1,1}(T_1)=\bigoplus_{\s\le \vell \le \s+\vdelta_{I}} \bigoplus_{0\le \x \le (\vdelta-\x_n)-\vdelta_I } U^{\vell}(\x)[-\sigma(\x)],\label{first equ}\\
		W_2&:={\psi}_{2,3-p_n}(T_2)=\bigoplus_{\s\le \vell \le \s+\vdelta_{I}} \bigoplus_{(p_n-2)\x_n\le \x \le \vdelta-\vdelta_I } U^{\vell}(\x)[-\sigma(\x)].\label{sencond equ}
		\end{align}
		Note that $\theta_1(\vdelta_{1_I})=\vdelta_{I}=\theta_2(\vdelta_{2_I})$.  The equality (\ref{first equ}) follows  from Lemma \ref{L-action to U}(b). For the equality (\ref{sencond equ}), we need to consider the cases according to whether $p_n$ is odd or even. Assume $p_n=2k+1$ for some $k\ge 1$. By Lemma \ref{L-action to U}(b), 
		\begin{align*}
			W_2&=((2k-2)\x_n)\psi_{2,0}(\bigoplus_{\s_2\le \vell \le \s_2+\vdelta_{2_I}} \bigoplus_{0\le \x \le \vdelta_{2}-\vdelta_{2_I}} U_2^{\vell}(\x-(k-1)\c_2)[-\sigma(\x)])    \\
			&=\bigoplus_{\s\le \vell \le \s+\vdelta_{I}} \bigoplus_{0\le \x \le \theta_2(\vdelta_2)-\vdelta_I } U^{\vell+(p_n-2)\x_n}(\x+2(k-1)\x_n-(k-1)\c)[-\sigma(\x)].
		\end{align*}
		Notice that  $U^{\vell+(p_n-2)\x_n}=U^{\vell}(\x_n)[-1]$ and $U^{\vell}(\c)=U^{\vell}[2]$ hold by Corollary \ref{x_i=[1]}. 
		Replacing $\x$ by $\x-(2k-1)\x_n=\x-(p_n-2)\x_n$, we obtain the equality (\ref{sencond equ}).
		
		On the other hand, assume $p_n=2k+2$ for some $k\ge 1$. By Lemma \ref{L-action to U}(b),
		\begin{align*}
			W_2&=((2k-1)\x_n)\psi_{2,0}(\bigoplus_{\s_2\le \vell \le \s_2+\vdelta_{2_I}} \bigoplus_{0\le \x \le \vdelta_{2}-\vdelta_{2_I}} U_2^{\vell}(\x+\x_{2,n}-k\c_2)[-\sigma(\x)])    \\
			&=\bigoplus_{\s\le \vell \le \s+\vdelta_{I}} \bigoplus_{0\le \x \le \theta_2(\vdelta_2)-\vdelta_I } U^{\vell}(\x+2k\x_n-k\c)[-\sigma(\x)].
		\end{align*}
		Replace $\x$ by $\x+2k\x_n=\x+(p_n-2)\x_n$, and we obtain the equality (\ref{sencond equ}) as claimed.
						 
		Observe that $[0,\vdelta-\x_n-\vdelta_{I}]\cup [(p_n-2)\x_n, \vdelta-\vdelta_{I}]=[0,\vdelta-\vdelta_{I}]$. Thus $V=W_1\oplus W_2$.
		Set $T_1':={\phi}_{2,4-p_n}{{\psi}_{1,1}}(T_1)={\phi}_{2,4-p_n}W_1$. Let $\y \in \L$ with $0\le \y \le (\vdelta-\x_n)-\vdelta_I$. Write $\vec t:=\y+(4-p_n)\x_n=\sum_{i=1}^n t_i\x_i+t\c$ in normal form. Then $t_n$ is one of the following three cases: $t_n=0$, $t_n=1$ and $t_n\ge 4$. 		
		Assume first that $t_n=0$. Then  $T_1'=0$ from the first case of  Lemma \ref{L-action to U}(a). 		
		Assume next that $t_n\ge 4$. Then  $T_1'=0$ from the third case of  Lemma \ref{L-action to U}(a). 	
		Assume  finally  that $t_n=1$. This means that $(p_n-3)\x_n\le \y \le (\vdelta-\x_n)-\vdelta_I$, and thus $\x_n \le \vec t \le \theta_2(\vdelta_2)+\x_n-\vdelta_I$. In this case, 
			\begin{align*}
				W_1((4-p_n)\x_n)=\bigoplus_{\s\le \vell \le \s+\vdelta_{I}} \bigoplus_{\x_n \le \vec t \le \theta_2(\vdelta_2)+\x_n-\vdelta_I } U^{\vell}(\vec t)[-\sigma(\vec t)+(4-p_n)].
			\end{align*}
		Following the second case of Lemma \ref{L-action to U}(a), we have 
			$$ T_1'= \bigoplus_{\s_2\le \vell \le \s_2+\vdelta_{2_I}} \bigoplus_{\x_{2,n} \le \vec x \le \vdelta_2+\x_{2,n}-\vdelta_{2_I} } U_2^{\vell}(\vec x-(4-p_n)\x_{2,n})[-\sigma(\vec x)+(4-p_n)].$$
			Since $p_{2,n}=2$,  we have $U_2^{\vell}(\x_{2,n})=U_2^{\vell}[1]$ by Corollary \ref{x_i=[1]}. Thus $$ T_1'= \bigoplus_{\s_2\le \vell \le \s_2+\vdelta_{2_I}} \bigoplus_{\x_{2,n} \le \vec x \le \vdelta_2+\x_{2,n}-\vdelta_{2_I} } U_2^{\vell}(\vec x)[-\sigma(\vec x)].$$
			Replace $\vec x$ by $\vec x+\x_{2,n}$, and then we have $T_1'= T_2$. 			
			In each case, the rigidity of $T_2$ implies that $\Hom_{\underline{\CM}^{\L_2} R^2}(T_2,T_1'[m])=0$ for any $m\neq 0$, and thus $V$ is a tilting  object in $\underline{\CM}^{\L}{R}$ by Proposition \ref{gluing tilting}. This finishes the induction step.
		
		\emph{Step $2$}: Concerning the shape of $V$, by	Lemma \ref{Hom}, we have that
		$$\underline{\Hom}(U^{\vell}(\x)[-\sigma(\x)],U^{\z}(\y)[-\sigma(\y)])\neq 0 \ \Longleftrightarrow \ \vell\ge \z \text{ and } 0\le \x-\y \le \s.$$ Moreover, in the case, the $\mathbf{k}$-space has dimension one. 
		Hence the endomorphism algebra $\underline{\End}(V)^{\rm op}\simeq \mathbf{k}Q_1/I_1 \otimes \mathbf{k}Q_2/I_2$. Here, each vertex $\vell$ in $Q_1$ (resp. $\x$ in $Q_2$) corresponds to  $U^{\vell}$ for $\s\le \vell \le \s+\vdelta_I$ (resp. $\rho(\mathbf{k})(\x)[-\sigma(\x)]$ for $0\le \x \le \vdelta-\vdelta_I$), and each arrow $x_i:\vell \to \vell+\x_i$ with $i\in I$ and $\vell,\vell+\x_i\in Q_1$ (resp. $x_i: \x \to \x+\x_i$ with  $i\in I^{c}$ and $\x,\x+\x_i\in Q_2$) corresponds the basis of one dimensional space $\underline{\Hom}(U^{\vell+\x_i},U^{\vell})$ (resp. $\underline{\Hom}(\rho(\mathbf{k})(\x+\x_i)[-\sigma(\x+\x_i)],\rho(\mathbf{k})(\x)[-\sigma(\x)])$), and the ideal $I_1:=(x_ix_j-x_jx_i\mid i,j \in I)$ (resp. $I_2:=(x_i^2, x_ix_j-x_jx_i\mid i,j \in I^{c})$).		
		Thus we have the isomorphism (\ref{shape of V}) from Lemma \ref{iso. algs} and Proposition \ref{shape theorem}.  
		
		\emph{Step $3$}: Concerning the last claim,		it suffices to refine the order of the index set $\{(\vell,\x)\mid \s\le \vell \le \s+\vdelta_I, 0\le \x \le \vdelta-\vdelta_I\}$ for indecomposable direct summands of the tilting object $V$ to a linear order, then yielding a full, strong exceptional sequence $E_1, E_2,\dots,E_m$ for $\underline{\CM}^{\L} R$, where $m:=\prod_{i=1}^{n}(p_i-1)$. 		
	\end{proof}

	As one extreme case $I=\{1,\dots,n\}$ of Theorem \ref{tilting object},  this immediately induces the following result,  
	which is known by Kussin-Lenzing-Meltzer \cite{KLM} for $n=3$, by Futaki-Ueda \cite{FU} and Herschend-Iyama-Minamoto-Oppermann \cite{HIMO} for general $n$. 
	
	\begin{corollary}\label{tilting n-cub}
		\cite[Theorem 4.76]{HIMO}  The object	$$T_{\rm cub}:=\bigoplus_{\s\leq\vell\leq\s+\vdelta}U^{\vell} $$ is   tilting  in $\underline{\CM}^{\L} R$, called the \emph{tilting $n$-cuboid}, with endomorphism algebra  $$\underline{\End}(T_{\rm cub})^{\rm op}\simeq 
		\bigotimes_{1\le i \le n} \mathbf{k}\vec{\AA}_{p_i-1}.$$
		Moreover, the objects $U^{\vell}$ $(\s\leq\vell\leq\s+\vdelta)$	can be arranged to form a full, strong exceptional sequence.		
	\end{corollary}			
	
	\begin{corollary} The endomorphism algebra $\underline{\End}(T_{\rm cub})^{\rm op}$ is isomorphic to the 
		incidence algebra of the poset $ [0,p_1-2]\times[0,p_2-2]\times\dots\times[0,p_n-2]$.
		
	\end{corollary}
	
	By using a similar argument as in the proof of Lemma \ref{Hom}, one can check easily that $\rho(\mathbf{k})(\x)$ $(0\le \x \le \vdelta)$ can be arranged to form a full exceptional sequence. However this is not strong in general. Thus it is natural to pose the following question, which is a higher version of
	the question posed by Kussin-Lenzing-Meltzer \cite[Remark 6.10]{KLM} in terms of Cohen-Macaulay modules.
	
	\begin{question} \label{question}
		Does there exist a full, strong exceptional sequence in $\underline{\CM}^{\L} R$ from the objects $\rho(\mathbf{k})(\x)$, $\x\in\L$.
	\end{question}
	
	There are some partial answers of Question \ref{question}, which are given by \cite{KLM,KLM2} for the type $(2,3,n)$ with $n \ge 2$, and by \cite{DR} for the type $(2,a,b)$ with $a,b\ge2$. 			
	As the other extreme case $I=\emptyset$ of Theorem \ref{tilting object}, we have the following result, where one of the implications gives a positive answer to Question \ref{question} for type $(2,p_2,\dots,p_n)$ with $p_i\ge 2$. In fact, we note that $p_1=2$ imples $\rho(\mathbf{k})[1]= \rho(\mathbf{k})(\x_1)$. 
	
	\begin{corollary} \label{tilting n-cub  2} The object
		$$T:=\bigoplus_{0\leq\x\leq\vdelta}\rho(\mathbf{k})(\x)[-\sigma(\x)]$$
		is   tilting  in $\underline{\CM}^{\L} R$ with endomorphism algebra $$\underline{\End}(T)^{\rm op}\simeq \bigotimes_{1\le i \le n} \mathbf{k}\vec{\AA}_{p_i-1}(2). $$ 
		Moreover, the objects $\rho(\mathbf{k})(\x)[-\sigma(\x)]$ $(0\le \x \le \vdelta)$	can be arranged to form a full, strong exceptional sequence.		
	\end{corollary}

	\section{Happel-Seidel symmetry} \label{sec: Happel-Seidel symmetry}
	Happel-Seidel symmetry, established in \cite{HS},  concerns the 
	derived equivalence of certain homogeneous acyclic Nakayama algebras. 
	A general and simple explanation of this symmetry was exhibited through the stable categories of vector bundles on  weighted projective lines in \cite{KLM}. In this section, we  generalize  this symmetry to the context of Brieskorn–Pham singularities. More precisely, we show that all these $(p_t-2)$-replicated algebras of $\bigotimes_{i\in S\setminus\{t\}}\mathbf{k}\vec{\AA}_{p_i-1}$ for $t\in S=\{1\dots,n\}$, are realized
	as endomorphism algebras of tilting objects on $\underline{\CM}^{\L} R$ of type $(p_1,\dots,p_n)$.

	Recall that for a finite dimensional algebra $\Lambda$, there exists an infinite dimensional locally 
	bounded algebra called \emph{repetitive algebra} $\widehat{\Lambda}$. The underlying vector space is given by $(\bigoplus_{i\in\Z}\Lambda)\oplus(\bigoplus_{i\in\Z}D\Lambda)$, with elements of the form $(a_i,f_i)_{i\in\Z}$ with finitely many $a_i$ and $f_i$ being nonzero.
	The multiplication is defined by $(a_i,f_i)_i(b_i,g_i)_i=(a_ib_i,a_{i+1}g_i+f_ib_i)_i$.
	The repetitive algebra $\widehat{\Lambda}$ of $\Lambda$ can interpreted as a $\Z\times\Z$ matrix algebra of the form
	\[\widehat{\Lambda}=\begin{pmatrix}
		\ddots & & & & & \\
		\ddots & \Lambda & & & \\
		& D\Lambda & \Lambda & &\\
		& & D\Lambda &\Lambda & \\
		& & & \ddots & \ddots 
	\end{pmatrix},
	\]
	see \cite{Hap1,S} for more details. For $m\geq 0$, the \emph{$m$-replicated algebra} $\Lambda^{(m)}$ of $\Lambda$ is the idempotent truncation $e^{(m)}\widehat{\Lambda}e^{(m)}$, where $e^{(m)}$ is the idempotent of $\widehat{\Lambda}$ given by the $\Z\times\Z$ matrix with $(i,i)$-th entry $1$ for all $i\in\{0,1,\ldots,m\}$ and zero everywhere else. In other words, we have 
				\begin{align*}
					\Lambda^{(m)}:=e^{(m)} \widehat{\Lambda}e^{(m)} =
						\begin{pmatrix} 
											\Lambda & & & & & \\
											D\Lambda & \Lambda & & & \\
											& D\Lambda & \ddots & &\\
											& & \ddots &\ddots & \\
											& & & D\Lambda & \Lambda 
										\end{pmatrix}
					\end{align*}
				where the matrix is of size $(m+1)\times (m+1)$. For $m=0$, we have ${\Lambda}^{(0)}=\Lambda	$. For $m=1$, ${\Lambda}^{(1)}$ is called the \emph{duplicated algebra}.	
	We refer to \cite{ABST1,ABST2, AI,CIM} for more information on $m$-replicated algebras.
		
	Let $R$ be an $\L$-graded BP singularity with weights $p_1,\ldots, p_n$. Let $t\in S:=\{1,\dots,n\}$ and $\Gamma^t:=\widehat{\Lambda_t}^{(p_t-2)}$, where $\Lambda_t:=\bigotimes_{i\in S\setminus\{t\}}\mathbf{k}\vec{\AA}_{p_i-1}$.

	\begin{lemma} \label{quiver of Gamma}
		The algebra $\Gamma^t$ is presented
		by the quiver $Q$ defined by 
			\begin{itemize}
			\item  $Q_0:=[\s+(p_t-2)\x_t,\s+\vdelta]\times [0,p_t-2]$,
			\item  $Q_1:=\{x_k:(\x,i)\to (\x+\x_k,i) \mid (\x,i),(\x+\x_k,i)\in Q_0\} $\\			
		 	$\cup\ \{ \prod_{k\in S\setminus\{t\}}x_k: (\s+\vdelta,i) \to (\s+(p_t-2)\x_t,i+1) \mid 0 \le i \le p_t-3 \},$		
		\end{itemize} subject to the relation
			\begin{itemize}
			\item $x_jx_k=x_kx_j$ and  $x_k^{p_k}=0$, where $j,k\in S\setminus\{t\}$.
			\end{itemize}
	\end{lemma}
	\begin{proof} 
		The quiver description of $\widehat{\Lambda_t}$ follows from the main result of \cite{S}. As an idempotent truncation,  we obtain the quiver description of $\Gamma^t$.
	\end{proof}
	\begin{example} (1) Assume $n=2$. Let $R$ be given by weight type $(p_1,p_2)$. Put $m:=(p_1-1)(p_2-1)$. 
		Then 
		$ \Gamma^1 \simeq \vec{\AA}_{m}(p_2)$ and $\Gamma^2\simeq\vec{\AA}_{m}(p_1)$.
	
	(2)	Assume $n=3$. Let $R$ be given by weight type $(p_1,p_2,p_3)=(3,4,5)$. Then the quivers of  $\Gamma^1$, $\Gamma^2$ and $\Gamma^3$ are the following, respectively.
	\begin{figure}[H]
		\centering 
		$\xymatrix@!R=2mm@!C=2mm{
			\bullet\ar[r]\ar[d]&\bullet\ar[r]\ar[d]&\bullet\ar[r]\ar[d]&\bullet\ar[d] & \bullet\ar[r]\ar[d]&\bullet\ar[r]\ar[d]&\bullet\ar[r]\ar[d]&\bullet\ar[d]\\
			\bullet\ar[r]\ar[d]&\bullet\ar[r]\ar[d]&\bullet\ar[r]\ar[d]&\bullet\ar[d]& \bullet\ar[r]\ar[d]&\bullet\ar[r]\ar[d]&\bullet\ar[r]\ar[d]&\bullet\ar[d] \\
			\bullet\ar[r]&\bullet\ar[r]&\bullet\ar[r]&\bullet\ar[ruu] & 
			\bullet\ar[r]&\bullet\ar[r]&\bullet\ar[r]&\bullet
		} $
		
		$$\xymatrix@!R=2mm@!C=2mm{
				\bullet\ar[r]\ar[d]&\bullet\ar[r]\ar[d]&\bullet\ar[r]\ar[d]&\bullet\ar[d] & \bullet\ar[r]\ar[d]&\bullet\ar[r]\ar[d]&\bullet\ar[r]\ar[d]&\bullet\ar[d] &
				\bullet\ar[r]\ar[d]&\bullet\ar[r]\ar[d]&\bullet\ar[r]\ar[d]&\bullet\ar[d]\\
				\bullet\ar[r]&\bullet\ar[r]&\bullet\ar[r]&\bullet\ar[ru] &
				\bullet\ar[r]&\bullet\ar[r]&\bullet\ar[r]&\bullet\ar[ru]  & 
				\bullet\ar[r]&\bullet\ar[r]&\bullet\ar[r]&\bullet \\
			} $$
		
		$$\xymatrix@!R=2mm@!C=2mm{
			\bullet\ar[r]\ar[d]&\bullet\ar[r]\ar[d]&\bullet\ar[d] & \bullet\ar[r]\ar[d]&\bullet\ar[r]\ar[d]&\bullet\ar[d]
			& \bullet\ar[r]\ar[d]&\bullet\ar[r]\ar[d]&\bullet\ar[d] & \bullet\ar[r]\ar[d]&\bullet\ar[r]\ar[d]&\bullet\ar[d]			\\
			\bullet\ar[r]&\bullet\ar[r]&\bullet\ar[ru]&\bullet\ar[r]&\bullet\ar[r]&\bullet\ar[ru] &
			\bullet\ar[r]&\bullet\ar[r]&\bullet\ar[ru]&\bullet\ar[r]&\bullet\ar[r]&\bullet
			} $$
		\caption{The quivers of $\Gamma^1$, $\Gamma^2$ and $\Gamma^3$ for $R(3,4,5)$}
	\end{figure}
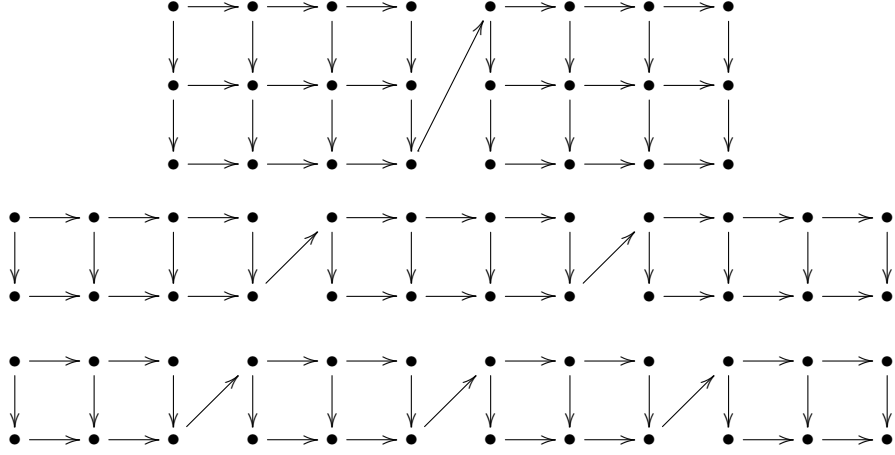	
	\end{example}

		Our main result in this section is the following. 
		Here,  note that $\mathbb{S}:=(\w)[n-2]=(-\s)[n]$ is the Serre functor on $\underline{\CM}^{\L} R$ by Theorem \ref{Auslander-Reiten-Serre duality}.
		
		\begin{theorem}	\label{replicated thm}The object
		$$T^t:=\bigoplus_{\s+(p_t-2)\x_t\le \vell \le \s+\vdelta}\bigoplus_{0\le i \le p_t-2} U^{\vell}(-i\s)[in]$$
		is  tilting in $\underline{\CM}^{\L} R$ with endomorphism algebra
		$$\underline{\End}(T^t)^{\rm op}\simeq \Gamma^t.$$
	\end{theorem} 
	
	\begin{proof}
		 We show that $T^n$ is a tilting object  in $\underline{\CM}^{\L} R$, and then by permuting the roles of weights $p_t$ and $p_n$,  the assertion follows.  We use induction on $n$.
		
		In the case $n=1$, by applying $(2-p_1)\x_1[p_1-2]$ to the special case $n=1$ of Corollary \ref{tilting n-cub  2}, we have that $T^1=\bigoplus_{i=0}^{p_1-2}\rho(\mathbf{k})(-i\x_1)[i]$ is tilting in $\underline{\CM}^{\L} R$.
		
		Consider the induction step and suppose that the assertion holds for all $k<n$. For $k=n$, it suffices to proceed by induction on $p_n$. If $p_n=2$, then the claim reduce to the case $k=n-1$ by Kn\"orrer periodicity. 
		We assume $p_n\ge 3$.  Let $R^{j}$ be an $\L_j$-graded BP singularity  with weights $p_1,\ldots,p_{n-1},p_{j,n}$ for $j=1,2$, where $p_{1,n}=p_n-1$ and $p_{2,n}=2$. 
		 By the induction hypothesis,  for $j=1,2$, the object 
		$$T_j= \bigoplus_{\s_j+(p_{j,n}-2)\x_{j,n}\le \vell \le \s_j+\vdelta_{j}} \bigoplus_{0\le i \le p_{j,n}-2} U_j^{\vell}(-i\s_j)[in] $$ 	
		is tilting in $\underline{\CM}^{\L_j}{R^j}$. 
			By Theorem \ref{ladder diagram},  we obtain a recollement 
			\begin{align*}
				\xymatrix{
					\underline{\CM}^{\L_1}{R^1}
					\ar[rrr]|{{\psi}_{1,p_n-2}\mathbb{S}}
					&&& \underline{\CM}^{\L}{R}
					\ar@/_1.5pc/[lll]|{\mathbb{S}^{-1}{\phi}_{1,p_n-2} }
					\ar@/^1.5pc/[lll]|{\mathbb{S}^{-1}{\phi}_{1,p_n-1}}
					\ar[rrr]|{{\phi}_{2,0}}
					&&& \underline{\CM}^{\L_2}{R^2},
					\ar@/_1.5pc/[lll]|{{\psi}_{2,-1}}
					\ar@/^1.5pc/[lll]|{{\psi}_{2,0}}
				}
			\end{align*}
			where $\mathbb{S}:=(\w)[n-2]=(-\s)[n]$.  
		By Lemma \ref{L-action to U}(b), we have the equalities 
		\begin{align*}
			V_1&:={{\psi}_{1,p_n-2}\mathbb{S}(T_1)}=\bigoplus_{\s+(p_n-2)\x_n\le \vell \le \s+\vdelta} \bigoplus_{1\le i \le p_n-2} U^{\vell}(-i\s)[in], \\
			V_2&:={\psi}_{2,0}(T_{2})=\bigoplus_{\s+(p_n-2)\x_n\le \vell \le \s+\vdelta}  U^{\vell}.
		\end{align*}
		Thus		$T^n=V_1\oplus V_2$. Following the second case of Lemma \ref{L-action to U}(a), we have 
		$$T_2':={\mathbb{S}^{-1}{\phi}_{1,p_n-2} (V_2)}=\bigoplus_{\s_1+(p_{1,n}-2)\x_{1,n}\le \vell \le \s_1+\vdelta_1} U_1^{\vell}(\s)[-n].$$
		Since $\mathbb{S}T_2'\in \add T_1$,  the rigidity of $T_1$  implies that $\Hom_{\underline{\CM}^{\L_1} R^1}(T_1,\mathbb{S}T_2'[m])=0$ for any $m\neq 0$. By Auslander-Reiten-Serre duality, we have $\Hom_{\underline{\CM}^{\L_1} R^1}(T_2',T_1[m])=0$ for any $m\neq 0$. Thus $T^n$ is tilting  in $\underline{\CM}^{\L}{R}$ by Proposition \ref{gluing tilting}.
		
		It remains to concern the shape of $T^n$. 
		Observe that for each $1\le k<n$ and $\s+(p_n-2)\x_n\le \vell\le \vell+\x_k \le \s+\vdelta$, we have $\underline{\Hom}(U^{\vell+\x_k},U^{\vell})=\mathbf{k}$ by Lemma \ref{Hom}. 
		For each $\s+(p_n-2)\x_n \le \vell\le \z \le \s+\vdelta$, by Auslander-Reiten-Serre duality, we get $D\underline{\Hom}(U^{\vell},U^{\z}(-\s)[n]) \simeq \underline{\Hom}(U^{\z},U^{\vell})=\mathbf{k}$. Representing each indecomposable direct summand $U^{\vell}(-i\s)[in]$ of $T^n$  by $(\vell, i)$, the quiver of $\underline{\End}(T^n)^{\rm op}$ coincides with the quiver of $\Gamma^n$. We are left to consider the space $\underline{\Hom}(U^{\vell},U^{\z}(-\s)[n])$ for $\z \not\ge\vell$, whose $\mathbf{k}$-dual, by using Auslander-Reiten-Serre duality, equals zero. Therefore we have shown that $\underline{\End}(T^n)^{\rm op}\simeq \Gamma^n$  by Lemma \ref{quiver of Gamma}. 				
	\end{proof}
	
	As an immediate consequence, we have the following result. 
	
	\begin{corollary} \label{nHappel-Seidel Symmetry}	
		For any $t\in \{1,\dots,n\}$, the $(p_t-2)$-replicated algebra $\Lambda_t$ and the stable category $\underline{\CM}^{\L} R$ are derived equivalent.
		
	\end{corollary}
	
	Combining Theorem \ref{tilting n-cub} and Corollary \ref{nHappel-Seidel Symmetry}, 
	the algebra $\bigotimes_{1\le i \le n} \mathbf{k}\vec{\AA}_{p_i-1}$ is 
	derived equivalent to the $(p_t-2)$-replicated algebra of $\bigotimes_{i\in S\setminus\{t\}}\mathbf{k}\vec{\AA}_{p_i-1}$ for some $t\in S=\{1\dots,n\}$.	
	In particular, as a special case $n=2$ of Corollary \ref{nHappel-Seidel Symmetry}, we obtain the following Happel-Seidel Symmetry \cite[Theorem 6.11]{KLM}, see also \cite[Corollary 1.2]{La}. 
	
	\begin{corollary}[Happel-Seidel Symmetry] \label{Happel-Seidel Symmetry n=2}   Let $a,b\ge 2$ and $m:=(a-1)(b-1)$. Then the algebras $\vec{\AA}_{m}(a)$, $\vec{\AA}_{m}(b)$ and the stable category $\underline{\CM}^{\L} R$ of  type $(a,b)$ are derived equivalent.
		
	\end{corollary}

	\section{Examples and applications}	\label{sec: Examples}
	In this section, we show how some of the main constructions work out in some particular and simple examples.
	
	\subsection{A simple example for gluing tilting objects}
	\begin{example}\label{example_1} 
		Let $n=2$ and assume that $(R^1,\L_1)$  and  $(R^2,\L_2)$ are GL hypersurface singularities of types  $(3,3)$ and $(3,2)$, respectively. In this case, we have $\vdelta_1=\s_1=\x_{1,1}+\x_{1,2}$, $\vdelta_2=\x_{2,1}$ and $\s_2=\x_{2,1}+\x_{2,2}$.  The tilting $2$-cuboids $T_1=\bigoplus_{\s_1 \le \vell \le \s_1+\vdelta_1} U_1^{\vell}$ in $\underline{\CM}^{\L_1}R^1$ and $T_2=\bigoplus_{\s_2 \le \vell \le \s_2+\vdelta_2} U_2^{\vell}$ in $\underline{\CM}^{\L_2}R^2$ can be presented in the component of the following Auslander-Reiten quivers, see \cite{HIMO, KLM}. Here, the label $\langle \z_j \rangle$  denotes the module $U_j^{\z_j}$ for $j=1,2$ and $\s_j\le \z_j \le \s_j+\vdelta_j$.
			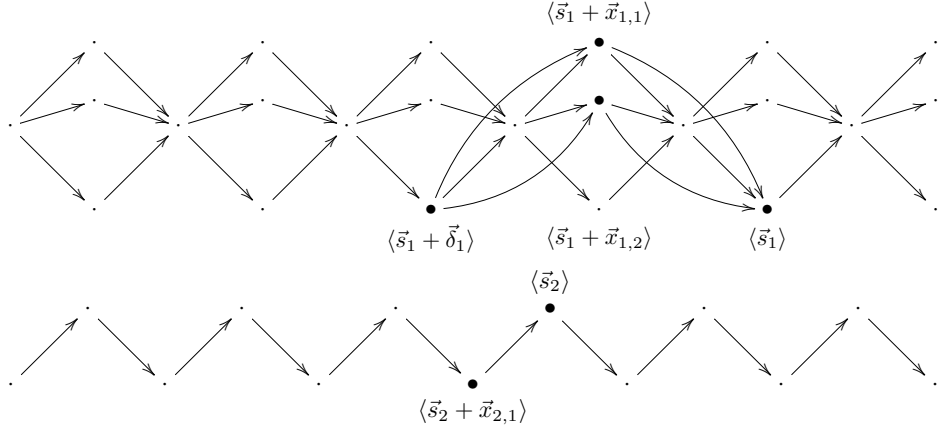
\begin{figure}[H]
			\centering 
			$
			\resizebox{\textwidth}{!}{
				\begin{xy} 0;<18pt,0pt>:<0pt,18pt>::
					(20,4) *+{\cdot} ="24",(22,6) *+{\cdot} ="34",(22,4.6) *+{\cdot} ="54",
					(22,2) *+{\cdot}="15",(30,1.3) *+{\langle \s_1+\vdelta_1\rangle} ,(24,4) *+{\cdot} ="25",
					(26,6) *+{\cdot} ="35", (26,4.6) *+{\cdot} ="55", (38,1.3) *+{\langle \s_1\rangle}, (34,1.3) *+{\langle \s_1+\x_{1,2}\rangle}, 
					(34,6.7) *+{\langle \s_1+\x_{1,1}\rangle}, 
					(26,2) *+{\cdot} ="16",(28,4) *+{\cdot} ="26",(30,6) *+{\cdot} ="36",(30,4.6) *+{\cdot} ="56",
					(30,2) *+{\bullet} ="17",(32,4) *+{\cdot} ="27",(34,6) *+{\bullet} ="37",(34,4.6) *+{\bullet} ="57",
					(34,2) *+{\cdot} ="18",(36,4) *+{\cdot} ="28",(38,6) *+{\cdot} ="38",(38,4.6) *+{\cdot} ="58",
					(38,2) *+{\bullet} ="19",(40,4) *+{\cdot} ="29",(42,6) *+{\cdot} ="39",(42,4.6) *+{\cdot} ="59",
					(42,2) *+{\cdot} ="110",
					"24", {\ar"34"},"24", {\ar"54"},"24", {\ar"15"},"34", {\ar"25"},"54", {\ar"25"},
					"15", {\ar"25"},"25", {\ar"35"},"25", {\ar"55"},"25", {\ar"16"},"35", {\ar"26"},"55", {\ar"26"},
					"16", {\ar"26"},"26", {\ar"36"},"26", {\ar"56"},"26", {\ar"17"},"36", {\ar"27"},"56", {\ar"27"},
					"17", {\ar"27"},"27", {\ar"37"},"27", {\ar"57"},"27", {\ar"18"},"37", {\ar"28"},"57", {\ar"28"},
					"18", {\ar"28"},"28", {\ar"38"},"28", {\ar"58"},"28", {\ar"19"},"38", {\ar"29"},"58", {\ar"29"},
					"19", {\ar"29"},"29", {\ar"39"},"29", {\ar"59"},"29", {\ar"110"},
					"17", {{\ar@[black]@/^1pc/@[r]"37"}},
					"17", {{\ar@[black]@/_1pc/@[r]"57"}},
					"37", {{\ar@[black]@/^1pc/@[r]"19"}},
					"57", {{\ar@[black]@/_1pc/@[r]"19"}},
				\end{xy}
			} 
			$ 
				$
			\resizebox{\textwidth}{!}{
				\begin{xy} 0;<16pt,0pt>:<0pt,16pt>::
					(18,2) *+{\cdot} ="14",(20,4) *+{\cdot} ="24",	
					(22,2) *+{\cdot} ="15",(24,4) *+{\cdot} ="25",
					(26,2) *+{\cdot} ="16",(28,4) *+{\cdot} ="26",
					(30,2) *+{\bullet} ="17", (30,1.3) *+{\langle\s_2+\x_{2,1}\rangle},
					(32,4) *+{\bullet} ="27", (32,4.7) *+{\langle\s_2\rangle},
					(34,2) *+{\cdot} ="18",(36,4) *+{\cdot} ="28",
					(38,2) *+{\cdot} ="19",(40,4) *+{\cdot} ="29",
					(42,2) *+{\cdot} ="110",
					"14", {\ar"24"},"24", {\ar"15"},
					"15", {\ar"25"},"25", {\ar"16"},
					"16", {\ar"26"},"26", {\ar"17"},
					"17", {\ar"27"},"27", {\ar"18"},
					"18", {\ar"28"},"28", {\ar"19"},
					"19", {\ar"29"},"29", {\ar"110"},
				\end{xy}
			}
			$			\caption{ The tilting $2$-cuboids   of types $(3,3)$ and $ (3,2)$ }	
		\end{figure}
	
	\end{example}
	
	We give a simple example to explain Theorem \ref{tilting object}.
	\begin{example} We continue to discuss Example \ref{example_1}.
		 Let $n=2$ and assume that $(R,\L)$ is a GL hypersurface singularity of type  $(3,4)$. In this case, we note that $\vdelta=\x_1+2\x_2$ and $\s=\x_1+\x_2$.
	
	(1) We show how the tilting objects $T_1=\bigoplus_{\s_1 \le \vell \le \s_1+\vdelta_1} U_1^{\vell}$ in  $\underline{\CM}^{\L_1} R^1$ and $T_2=\bigoplus_{\s_2 \le \vell \le \s_2+\vdelta_2} U_2^{\vell}$ in $\underline{\CM}^{\L_2} R^2$ glue a tilting object in $\underline{\CM}^{\L} R$. By Theorem \ref{ladder diagram}, we obtain a recollement
	\begin{align*}
		\xymatrix{
			\underline{\CM}^{\L_1}{R^1}
			\ar[rrr]|{{\psi}_{1,2}}
			&&& \underline{\CM}^{\L}{R}
			\ar@/_1.5pc/[lll]|{{\phi}_{1,2} }
			\ar@/^1.5pc/[lll]|{{\phi}_{1,3}}
			\ar[rrr]|{{\phi}_{2,0}}
			&&& \underline{\CM}^{\L_2}{R^2}.
			\ar@/_1.5pc/[lll]|{{\psi}_{2,-1}}
			\ar@/^1.5pc/[lll]|{{\psi}_{2,0}}
		}
	\end{align*}
	By Lemma \ref{L-action to U}(b),  ${\psi}_{1,2}(T_1)=\bigoplus_{\s \le \vell\le \s+\x_1+\x_2}U^{\vell}$ and $\psi_{2,0}(T_2)= U^{\s+2\x_2}\oplus U^{\s+\vdelta}$. By the second case of Lemma \ref{L-action to U}(a), we have 
	$T_2':={\phi}_{1,2}\psi_{2,0}(T_2)=U_1^{\s_1+\x_{1,1}}\oplus U_1^{\s_1+\vdelta_1}$, which belongs to $ \add T_1$. The rigidity of $T_1$ implies that $\Hom_{\underline{\CM}^{\L_1} R^1}(T_2',T_1[m])=0$ for any $m\neq 0$, and thus $\bigoplus_{\s \le \vell\le \s+\vdelta} U^{\vell}$ is tilting  in $V_1=\underline{\CM}^{\L}{R}$ by  Proposition \ref{gluing tilting}.
	
	(2) 	Notice that $T_1=\bigoplus_{0 \le \x \le \vdelta_1} \rho(\mathbf{k})(\x)[-\sigma(\x)]$ and $T_2=\bigoplus_{0 \le \x \le \vdelta_2} \rho(\mathbf{k})(\x)[-\sigma(\x)]$ by  Corollary \ref{x_i=[1]}. We construct another tilting object in $\underline{\CM}^{\L} R$. By Theorem \ref{ladder diagram}, we obtain a piece of ladder 
	\begin{equation*}
		\xymatrix{
			\underline{\CM}^{\L_1}{R^1}
			\ar[rrr]|{{\psi}_{1,1}}
			\ar@/_3pc/[rrr]|{{\psi}_{1,2}}
			&&& \underline{\CM}^{\L}{R}
			\ar@/_1.5pc/[lll]|{{\phi}_{1,1} }
			\ar@/^1.5pc/[lll]|{{\phi}_{1,2}}
			\ar[rrr]|{{\phi}_{2,-1}}
			\ar@/_3pc/[rrr]|{{\phi}_{2,0}}
			&&& \underline{\CM}^{\L_2}{R^2}.
			\ar@/_1.5pc/[lll]|{{\psi}_{2,-2}}
			\ar@/^1.5pc/[lll]|{{\psi}_{2,-1}}
		}
	\end{equation*} 	
	By Lemma \ref{L-action to U}(b) and Corollary \ref{x_i=[1]}, we have 
	\begin{align*}
		{\psi}_{1,1}(T_1)&=\bigoplus_{0\le \x \le \x_1+\x_2}\rho(\mathbf{k})(\x)[-\sigma(\x)]=U^{\s}\oplus U^{\s+\x_1} \oplus U^{\s+2\x_2} \oplus U^{\s+\vdelta},\\
		{\psi}_{2,-1}(T_2)&= \bigoplus_{2\x_2\le \x \le \x_1+2\x_2}\rho(\mathbf{k})(\x)[-\sigma(\x)]=U^{\s}(2\x_2)[-2]\oplus U^{\s+\x_1}(2\x_2)[-2].
	\end{align*}
	Following the first case of Lemma \ref{L-action to U}(a), we have $T_1':=\phi_{2,0}{\psi}_{1,1}(T_1)=T_2$. The rigidity of $T_2$ implies that $\Hom_{\underline{\CM}^{\L_2} R^2}(T_2,T_1'[m])=0$ for any $m\neq 0$, and thus $V_2=\bigoplus_{0 \le \x \le \vdelta} \rho(\mathbf{k})(\x)[-\sigma(\x)]$ is a tilting  object in $\underline{\CM}^{\L}{R}$ by Proposition \ref{gluing tilting}.
	
	The quiver of the tilting object $V_1$ (resp. $V_2$) in $\underline{\CM}^{\L}{R}$ is depicted  the points $\bullet$ together with the blue (resp. red) arrows in the following Auslander-Reiten quiver. Here, the label $\langle \z \rangle$  denotes the module $U^{\z}$ for $\s\le \z \le \s+\vdelta$.

		\begin{figure}[H]
			$
		\resizebox{\textwidth}{!}{
			\begin{xy} 0;<19pt,0pt>:<0pt,20pt>::
				(13,7) *+{\cdot} ="73",(13,5) *+{\cdot} ="54",(14,6) *+{\cdot} ="64",(15,7) *+{\cdot} ="74",(14,5) *+{\cdot} ="84",
				(13,3) *+{\cdot} ="35",(14,4) *+{\cdot} ="45",(15,5) *+{\cdot} ="55",(16,6) *+{\cdot} ="65",(17,7) *+{\cdot} ="75",
				(16,5) *+{\cdot} ="85",(15,3) *+{\bullet} ="36",(15.3,2.3) *+{\langle\s+\x_{1}\rangle(2\x_2)[-2]} , (16,4) *+{\cdot} ="46",(17,5) *+{\cdot} ="56",(18,6) *+{\cdot} ="66",(19,7) *+{\bullet} ="76", (19,7.7) *+{\langle \s\rangle(2\x_2)[-2]},  (18,5) *+{\cdot} ="86",
				(17,3) *+{\cdot} ="37",(18,4) *+{\cdot} ="47",(19,5) *+{\cdot} ="57",(20,6) *+{\cdot} ="67",(21,7) *+{\cdot} ="77",
				(20,5) *+{\cdot} ="87",(19,3) *+{\cdot} ="38",(20,4) *+{\cdot} ="48",(21,5) *+{\cdot} ="58",(22,6) *+{\cdot} ="68",
				(23,7) *+{\cdot} ="78",(22,5) *+{\cdot} ="88",(21,3) *+{\bullet} ="39", (21,2.3) *+{\langle \s+\vdelta\rangle},
				(22,4) *+{\cdot} ="49",(23,5) *+{\cdot} ="59",(24,6) *+{\cdot} ="69",(25,7) *+{\bullet} ="79",(25,7.7) *+{\langle \s+2\x_{2}\rangle},(24,5) *+{\bullet} ="89", (24,2.3) *+{\langle \s+\x_{1}+\x_{2}\rangle},
				(23,3) *+{\cdot} ="310",(24,4) *+{\cdot} ="410",(25,5) *+{\cdot} ="510",(26,6) *+{\cdot} ="610",(27,7) *+{\cdot} ="710",
				(26,5) *+{\cdot} ="810",(25,3) *+{\cdot} ="311",(26,4) *+{\cdot} ="411",(27,5) *+{\cdot} ="511",(28,6) *+{\cdot} ="611",
				(29,7) *+{\cdot} ="711",(28,5) *+{\bullet} ="811", (28,7.7) *+{\langle \s+\x_{2}\rangle},
				(27,3) *+{\bullet} ="312",(27,2.3) *+{\langle \s+\x_{1}\rangle}, (28,4) *+{\cdot} ="412",(29,5) *+{\cdot} ="512",(30,6) *+{\cdot} ="612", (31,7) *+{\bullet} ="712", (31,7.7) *+{\langle \s\rangle} ,(30,5) *+{\cdot} ="812",
				(29,3) *+{\cdot} ="313",(30,4) *+{\cdot} ="413",(31,5) *+{\cdot} ="513",(32,6) *+{\cdot} ="613",(33,7) *+{\cdot} ="713",
				(32,5) *+{\cdot} ="813",(31,3) *+{\cdot} ="314",(32,4) *+{\cdot} ="414",(33,5) *+{\cdot} ="514",(33,3) *+{\cdot} ="315",
				"73", {\ar"64"},
				"54", {\ar"64"},"64", {\ar"74"},"54", {\ar"84"},"54", {\ar"45"},"64", {\ar"55"},"74", {\ar"65"},"84", {\ar"55"},
				"35", {\ar"45"},"45", {\ar"55"},"55", {\ar"65"},"65", {\ar"75"},"55", {\ar"85"},"45", {\ar"36"},
				"55", {\ar"46"},"65", {\ar"56"},"75", {\ar"66"},"85", {\ar"56"},
				"36", {\ar"46"},"46", {\ar"56"},"56", {\ar"66"},"66", {\ar"76"},"56", {\ar"86"},"46", {\ar"37"},
				"56", {\ar"47"},"66", {\ar"57"},"76", {\ar"67"},"86", {\ar"57"},
				"37", {\ar"47"},"47", {\ar"57"},"57", {\ar"67"},"67", {\ar"77"},"57", {\ar"87"},
				"47", {\ar"38"},"57", {\ar"48"},"67", {\ar"58"},"77", {\ar"68"},"87", {\ar"58"},
				"38", {\ar"48"},"48", {\ar"58"},"58", {\ar"68"},"68", {\ar"78"},
				"58", {\ar"88"},"48", {\ar"39"},"58", {\ar"49"},"68", {\ar"59"},"78", {\ar"69"},"88", {\ar"59"},
				"39", {\ar"49"},"49", {\ar"59"},"59", {\ar"69"},"69", {\ar"79"},"59", {\ar"89"},
				"49", {\ar"310"},"59", {\ar"410"},"69", {\ar"510"},"79", {\ar"610"},"89", {\ar"510"},
				"310", {\ar"410"},"410", {\ar"510"},"510", {\ar"610"},"610", {\ar"710"},"510", {\ar"810"},
				"410", {\ar"311"},"510", {\ar"411"},"610", {\ar"511"},"710", {\ar"611"},"810", {\ar"511"},
				"311", {\ar"411"},"411", {\ar"511"},"511", {\ar"611"},"611", {\ar"711"},"511", {\ar"811"},
				"411", {\ar"312"},"511", {\ar"412"},"611", {\ar"512"},"711", {\ar"612"},"811", {\ar"512"},
				"312", {\ar"412"},"412", {\ar"512"},"512", {\ar"612"},"612", {\ar"712"},"512", {\ar"812"},
				"412", {\ar"313"},"512", {\ar"413"},"612", {\ar"513"},"712", {\ar"613"},"812", {\ar"513"},
				"313", {\ar"413"},"413", {\ar"513"},"513", {\ar"613"},	"613", {\ar"713"},"513", {\ar"813"},
				"413", {\ar"314"},"513", {\ar"414"},"613", {\ar"514"},"813", {\ar"514"},"314", {\ar"414"},
				"414", {\ar"514"},"414", {\ar"315"},
				"36", {{\ar@[red]@/^1pc/@[r]"76"}},
				"36", {{\ar@[red]@/^1pc/@[r]"39"}},
				"76", {{\ar@[red]@/_1pc/@[r]"79"}},
				"79", {{\ar@[red]@/_1pc/@[r]"712"}},
				"39", {{\ar@[red]@/^1pc/@[r]"312"}},
				"312", {{\ar@[red]@/^1pc/@[r]"712"}},
				"39", {{\ar@[blue]@/^1pc/@[r]"89"}},
				"89", {{\ar@[blue]@/^1pc/@[r]"312"}},
				"312", {{\ar@[blue]@/_1pc/@[r]"712"}},
				"39", {{\ar@[blue]@/^1pc/@[r]"79"}},
				"79", {{\ar@[blue]@/_1pc/@[r]"811"}},
				"811", {{\ar@[blue]@/_1pc/@[r]"712"}},
			\end{xy}
		}
		$\caption{ Two of the extended tilting $2$-cuboids of type $(3,4)$}	
		\end{figure}
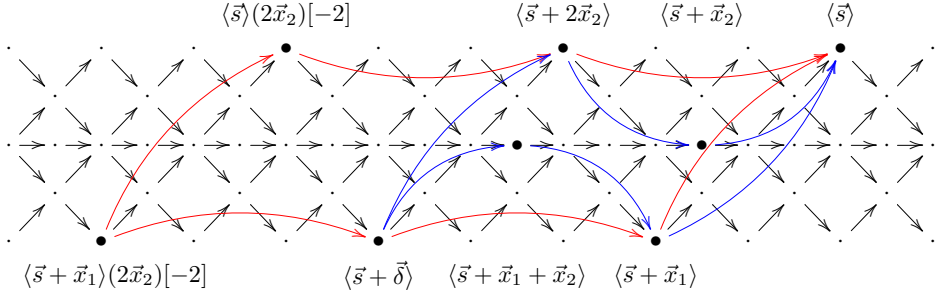	
		We remark that the endomorphism algebra $\underline{\End} (V_2)^{\rm op}$ is isomorphism to $\Lambda(\mathbf{q})$ for $\mathbf{q}=(2,2)$ as in the Figure \ref{figure (2,2)}, that is, $\mathbf{k}\vec{\AA}_2\otimes\mathbf{k}\vec{\AA}_3(2)$.
%
	\end{example}
		
		\subsection{Replicated algebras of the Dynkin type} In this subsection, we present some derived equivalences for  replicated algebras of the Dynkin type.
						
		Let us first recall the following well-known result in \cite{KLM}.
		\begin{lemma} \label{tensor product}
			\begin{itemize}
				\item[(a)] For $m=2,3,4$, the algebra $\mathbf{k}\vec{\AA}_2\otimes\mathbf{k}\vec{\AA}_m$ is derived equivalent to
				$\mathbf{k}\DD_4$ if $m=2$, $\mathbf{k}\EE_6$ if $m=3$, and $\mathbf{k}\EE_8$ if $m=4$.
				\item[(b)] $\mathbf{k}\vec{\AA}_2\otimes\mathbf{k}\vec{\AA}_5$ is derived equivalent to the tubular canonical algebra of type $(2,3,6)$.
				\item[(c)] $\mathbf{k}\vec{\AA}_3\otimes\mathbf{k}\vec{\AA}_3$ is derived equivalent to the tubular canonical algebra of type $(2,4,4)$.
				\item[(d)] $\mathbf{k}\vec{\AA}_2\otimes\mathbf{k}\vec{\AA}_2\otimes\mathbf{k}\vec{\AA}_2$ is derived equivalent to the tubular canonical algebra of type $(3,3,3)$.
			\end{itemize}
		\end{lemma}
		
		Note that if two finite dimensional algebras $A$ and $B$ are derived equivalent, then  their repetitive algebras $\widehat{A}$ and $\widehat{B}$ are derived equivalent (see \cite{As}), and also their $m$-replicated algebras $A^{(m)}$ and $B^{(m)}$, see \cite{CL}.
		
		\begin{proposition}  \label{tensor product 2}
			Let $l,m \ge 2$. Then the following assertions hold. 
			\begin{itemize} 
				\item[(a)] $($\cite[Theorem 6.11]{KLM}, \cite[Corollary 1.2]{La}$)$ The algebra $\mathbf{k}\vec{\AA}_l\otimes\mathbf{k}\vec{\AA}_m$ is derived equivalent to the $(l-1)$-replicated algebra of $\mathbf{k}\vec{\AA}_m$, that is,  $\mathbf{k}\vec{\AA}_{lm}(m+1)$.
				\item[(b)]  The algebra $\mathbf{k}\vec{\AA}_2\otimes\mathbf{k}\vec{\AA}_l\otimes\mathbf{k}\vec{\AA}_m$ is derived equivalent to the  $(l-1)$-replicated algebra of $\mathbf{k}\DD_4$ if $m=2$, that of $\mathbf{k}\EE_6$ if $m=3$, and that of $\mathbf{k}\EE_8$ if $m=4$.
			\end{itemize}
		\end{proposition}
		\begin{proof} (a) This follows directly from Corollary \ref{Happel-Seidel Symmetry n=2}.
			
			(b) By Corollary \ref{nHappel-Seidel Symmetry}, the algebra $\mathbf{k}\vec{\AA}_2\otimes\mathbf{k}\vec{\AA}_l\otimes\mathbf{k}\vec{\AA}_m$ is derived equivalent to the  $(l-1)$-replicated algebra of $\mathbf{k}\vec{\AA}_2\otimes\mathbf{k}\vec{\AA}_m$. The remaining assertion follows directly from Lemma \ref{tensor product}(a).
		\end{proof}
		
		As a consequence, we have the following result.
		
		\begin{corollary} \label{tensor product 3}For any $l,m,t \ge 2$, the algebra $\mathbf{k}\vec{\AA}_l\otimes\mathbf{k}\vec{\AA}_m\otimes\mathbf{k}\vec{\AA}_t$ is derived equivalent to the $(t-1)$-replicated algebra of the Nakayama algebra $\mathbf{k}\vec{\AA}_{lm}(m+1)$.
			\end{corollary}
		
		\begin{example} By Lemma \ref{tensor product}, the tubular canonical algebra  of type $(3,3,3)$ is derived equivalent to the algebra $(\mathbf{k}\vec{\AA}_2)^{\otimes 3}$. It is presented by the following quiver
		
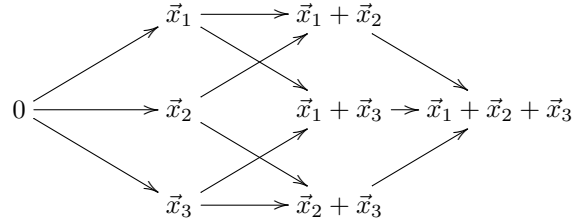
\[
		\begin{xy} 0;<3pt,0pt>:<0pt,3.6pt>:: 
			(-30,0) *+{0} ="0",
			(-10,10) *+{\x_1} ="1",(-10,0) *+{\x_2} ="2",(-10,-10) *+{\x_3} ="3",
			(10,10) *+{\x_1+\x_2} ="12",(10,0) *+{\x_1+\x_3} ="13",(10,-10) *+{\x_2+\x_3} ="23",
			(30,0) *+{\x_1+\x_2+\x_3} ="123",
			"0", {\ar"1"},"0", {\ar"2"},"0", {\ar"3"},
			"1", {\ar"12"},"1", {\ar"13"},"2", {\ar"12"},"2", {\ar"23"},"3", {\ar"13"},"3", {\ar"23"},
			"12", {\ar"123"},"13", {\ar"123"},"23", {\ar"123"},
		\end{xy}
		\]
		with the commutativity relations at each square. 
		It is further  derived equivalent to the $1$-replicated algebras  $((\mathbf{k}\vec{\AA}_2)^{\otimes 2})^{(1)}$ by Corollary \ref{nHappel-Seidel Symmetry},  $\mathbf{k}\DD_4^{(1)}$ by Proposition \ref{tensor product 2},
		and $(\mathbf{k}\vec {\AA}_4(3))^{(1)}$ by Corollary \ref{tensor product 3}. These algebras are presented by the following quivers with relations, respectively. Here,  the dotted lines indicate the relations.	
		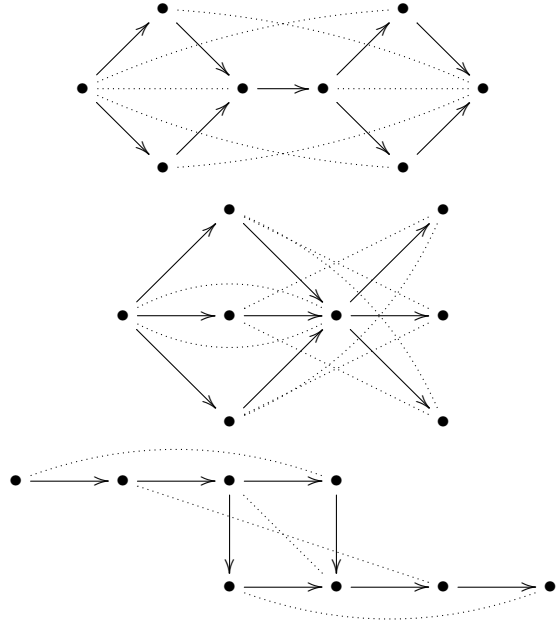
\begin{figure}[H]
		\centering 
		
		$
			\begin{xy} 0;<3pt,0pt>:<0pt,3pt>:: 
			(-20,0) *+{\bullet} ="0",
			(-10,10) *+{\bullet} ="1",(-10,-10) *+{\bullet} ="3",
			(20,10) *+{\bullet} ="12",(0,0) *+{\bullet} ="13",(20,-10) *+{\bullet} ="23",
			(10,0) *+{\bullet} ="123",(30,0) *+{\bullet} ="14",
			"0", {\ar"1"},"0", {\ar"3"},"1", {\ar"13"},"3", {\ar"13"},
			"13", {\ar"123"},"123", {\ar"12"},"123", {\ar"23"}, "23", {\ar"14"},"12", {\ar"14"},
			"0", {{\ar@{.}"13"}},
			"0", {{\ar@[black]@/^0.5pc/@{.}"12"}},"0", {{\ar@[black]@/_0.5pc/@{.}"23"}},
			"123", {{\ar@{.}"14"}},
			"1", {{\ar@[black]@/^0.5pc/@{.}"14"}},"3", {{\ar@[black]@/_0.5pc/@{.}"14"}},
		\end{xy}
		$
				\[
			\begin{xy} 0;<4pt,0pt>:<0pt,4pt>:: 
				(-20,0) *+{\bullet} ="0",
				(-10,10) *+{\bullet} ="1",(-10,0) *+{\bullet} ="2",(-10,-10) *+{\bullet} ="3",
				(10,10) *+{\bullet} ="12",(0,0) *+{\bullet} ="13",(10,-10) *+{\bullet} ="23",
				(10,0) *+{\bullet} ="123",
				"0", {\ar"1"},"0", {\ar"2"},"0", {\ar"3"},
				"1", {\ar"13"},"2", {\ar"13"},"3", {\ar"13"},
				"13", {\ar"123"},"13", {\ar"12"},"13", {\ar"23"},
				"0", {{\ar@[black]@/^1pc/@{.}"13"}},"0", {{\ar@[black]@/_1pc/@{.}"13"}},
				"1", {{\ar@{.}"123"}},"1", {{\ar@[black]@/^1pc/@{.}"23"}},
				"3", {{\ar@{.}"123"}},"3", {{\ar@[black]@/_1pc/@{.}"12"}},
				"2", {{\ar@{.}"12"}},	"2", {{\ar@{.}"23"}}
			\end{xy}
			\]
			$
			\begin{xy} 0;<4pt,0pt>:<0pt,4pt>:: 
				(20,0) *+{\bullet} ="0", (0,10) *+{\bullet} ="2",  (30,0) *+{\bullet} ="1",(10,10) *+{\bullet} ="12",
				(40,0) *+{\bullet} ="11",(20,10) *+{\bullet} ="112",(30,10) *+{\bullet} ="111",(50,0) *+{\bullet} ="1112",
				"2", {\ar"12"}, "12", {\ar"112"},"112", {\ar"111"},"112", {\ar"0"},"0", {\ar"1"},"1", {\ar"11"},"11", {\ar"1112"},"111", {\ar"1"},
				"2", {{\ar@[black]@/^1pc/@{.}"111"}},"0", {{\ar@[black]@/_1pc/@{.}"1112"}},	
				"12", {{\ar@{.}"11"}},"112", {{\ar@{.}"1"}}	
			\end{xy}
			$
			
			\caption{ The $1$-replicated algebras  $((\mathbf{k}\vec{\AA}_2)^{\otimes 2})^{(1)}$, $\mathbf{k}\DD_4^{(1)}$ and $(\mathbf{k}\vec {\AA}_4(3))^{(1)}$}	
			\end{figure}

	\end{example}
		
		\noindent {\bf Acknowledgements.}  The author would like to thank Jianmin Chen and Shiquan Ruan for helpful comments and discussions.

			\vskip 5pt
		\noindent {\scriptsize   \noindent  Weikang Weng\\
			School of Mathematical Sciences, \\
			Xiamen University, Xiamen, 361005, Fujian, PR China.\\
			E-mails: 
			wkweng@stu.xmu.edu.cn\\ }
		\vskip 3pt
		
	\end{document}